\documentclass[12pt,twoside,leqno]{article}
\usepackage{amsthm,amsfonts,amssymb,amsmath, amscd, enumerate}
\usepackage[colorlinks=true]{hyperref}
\usepackage[mathscr]{eucal}
\usepackage{mathrsfs}
\usepackage[all]{xy}
\usepackage{comment}
\usepackage{tikz-cd} 
\theoremstyle{plain}

\newcommand{\CJ}{{\mathcal {J}}}

\newcommand{\CO}{{\mathcal {O}}}

\newcommand{\CS}{{\mathcal {S}}}

\newcommand{\CX}{{\mathcal {X}}}
\newcommand{\CY}{{\mathcal {Y}}}
\newcommand{\CZ}{{\mathcal {Z}}}

\newcommand{\Alb}{{\mathrm{Alb}}}
\newcommand{\Aut}{{\mathrm{Aut}}}
\newcommand{\Char}{{\mathrm{char}}}
\newcommand{\Br}{{\mathrm{Br}}}

\renewcommand{\div}{{\mathrm{div}}}
\newcommand{\red}{{\mathrm{red}}}

\newcommand{\Gal}{{\mathrm{Gal}}}

\newcommand{\Frac}{{\mathrm{Frac}}}
\newcommand{\Hom}{{\mathrm{Hom}}}

\newcommand{\Ker}{{\mathrm{Ker}}}

\newcommand{\NS}{{\mathrm{NS}}}

\newcommand{\non}{{\mathrm{non-}}}

\newcommand{\Pic}{\mathrm{Pic}}

\renewcommand{\Im}{{\mathrm{Im}}}

\font\cyr=wncyr10  \newcommand{\Sha}{\hbox{\cyr X}}

\newcommand{\cris}{{\mathrm{cris}}}

\newcommand{\syn}{{\mathrm{syn}}}
\newcommand{\tor}{{\mathrm{tor}}}

\newcommand{\Coker}{{\mathrm{Coker}}}

\DeclareMathOperator{\Spec}{Spec}

\newcommand{\QQ}{\mathbb{Q}}

\newcommand{\ZZ}{\mathbb{Z}} 
 
\newcommand{\GG}{\mathbb{G}} 
\newcommand{\LL}{\mathbb{L}} 
\newcommand{\HH}{\mathbb{H}}

\newcommand{\lra}{\longrightarrow}

\newcommand{\fppf}{\mathrm{fppf}} 
\newcommand{\et}{\mathrm{et}} 
\newcommand{\zar}{\mathrm{Zar}}

\newcommand{\SO}{{\mathscr{O}}}

\newcommand{\cH}{{\mathrm{CH}}}
\newcommand{\olsi}[1]{\,\overline{\!{#1}}}
\newcommand{\ra}{\rightarrow}

\newtheorem{thm}{Theorem}[section]
\newtheorem{cor}[thm]{Corollary}
\newtheorem{lem}[thm]{Lemma}
\newtheorem{prop}[thm]{Proposition}
\newtheorem{conj}[thm]{Conjecture}
\newtheorem{defn}[thm]{Definition}

\theoremstyle{definition}

\theoremstyle{remark}
\newtheorem{rem}[thm]{Remark}

\newcommand{\Addresses}{{
  \bigskip
  \footnotesize

  \textsc{Fakultät für Mathematik, Universität Regensburg, 93040 Regensburg, Germany}\par\nopagebreak
  \textit{E-mail address}: \texttt{yanshuai.qin@ur.de}

  }}
\begin{document}
\date{}

\title{On the Brauer groups of fibrations II}

\author{Yanshuai Qin}
\renewcommand{\thefootnote}{\fnsymbol{footnote}} 
\footnotetext{\emph{MSC classes}: 11G40, 14G17, 14J20, 11G25, 11G35}
\footnotetext{The paper was written when the author was in his Ph.D program in UC Berkeley and was supported by several summer grants from UCB. The author is currently supported by the DFG through CRC1085 Higher Invariants (University of Regensburg).}
\maketitle
\begin{abstract}
    Let $K$ be a number field, and let $\CX$ be a proper regular flat scheme over $\CO_K$ with a generic fiber $X$ geometrically connected over $K$. We prove that there is an exact sequence up to finite groups $0\lra \Sha(\Pic^0_{X/K})\lra \Br(X)\lra \Br(X_{\bar{K}})^{G_K}\lra 0$, which generalizes a theorem of Artin and Grothendieck for arithmetic surfaces to arbitrary dimensions. Consequently, we reduce Artin's question regarding the finiteness of $\Br(\CX)$ for proper regular flat schemes $\CX$ over $\ZZ$ to $3$-dimensional arithmetic schemes.
\end{abstract}
\tableofcontents
\section{Introduction}
Let $C$ be the spectrum of the ring of integers in a number field or a smooth projective geometrically connected curve over a finite field with function field $K$. Let $\CX$ be a $2$-dimensional regular scheme and let $\pi:\mathcal{X}\longrightarrow C$ be a proper flat morphism such that the generic fiber $X$ of $\pi$ is smooth and geometrically connected over $K$.  Artin and Grothendieck (cf.\cite[$\S$ 4]{Gro3} or \cite[Prop. 5.3]{Ulm}) proved that there is an isomorphism 
$$
\Sha(\Pic^{0}_{X/K})\cong \Br(\mathcal{X})
$$
up to finite groups. Assuming the finiteness of $\Sha(\Pic^{0}_{X/K})$ or $\Br(\mathcal{X})$, Milne \cite{Mil4}, Gonzales-Aviles \cite{Goa}, Liu-Lorenzini-Raynaud \cite{LLR1, LLR2}, and Geisser \cite{Gei1} established a precise relationship between their orders. In \cite{Qin1}, the author proved a relation between Tate-Shafarevich groups and geometric Brauer groups for fibrations over arbitrary finitely generated fields. The aim of this article is to generalize the original result of Artin and Grothendieck to fibrations of arbitrary relative dimensions over the spectrum of the ring of integers in a number field.

For any noetherian scheme $X$,  the \emph{cohomological Brauer group}
$$
\Br(X):=H^2(X,\mathbb{G}_m)_{\tor}
$$
is defined to be the torsion part of the etale cohomology group $H^2(X,\GG_m)$.
\subsection*{Tate conjecture for divisors}
Let us first recall the Tate conjecture for divisors over finitely generated fields.
\begin{conj} $($Conjecture $T^1(X,\ell))$.
Let $X$ be a projective and smooth variety over a finitely generated field $k$ of characteristic $p\geq0$, and let $\ell\neq p$ be a prime number. Then the cycle class map\\
$$
\Pic(X)\otimes_\mathbb{Z}\mathbb{Q}_\ell\longrightarrow H^2(X_{k^s},\mathbb{Q}_\ell(1))^{G_k}
$$
is surjective.
\end{conj}
It is well-known that $T^1(X,\ell)$ is equivalent to the finiteness of $\Br(X_{k^s})^{G_k}(\ell)$ (cf. Proposition \ref{prop2.1}).
\subsection*{Artin's question}
Artin conjectured that $\Br(\mathcal{X})$ is finite for any proper scheme $\CX$ over $\Spec\mathbb{Z}$. For a smooth projective variety $\CX$ over a finite field, it is well-known that the finiteness of $\Br(\mathcal{X})$ is equivalent to the Tate conjecture for divisors on $\CX$ (cf. \cite{Tat2}). For a proper flat regular integral scheme $\CX$ over $\Spec \ZZ$, taking $C=\Spec \Gamma (\CX,\CO_\CX)$, there exists a proper flat morphism (the Stein factorization) $\pi:\CX \lra C$  with a generic fiber geometrically connected over $K=K(C)$. We will show that our main theorem implies that the finiteness of $\Br(\mathcal{X})(\ell)$ is equivalent to the finiteness of $\Sha(\Pic^0_{X/K})(\ell)$ and $T^1(X,\ell)$ for the generic fiber $X$ of  $\pi$. In \cite{SZ}, Skorobogatov and Zarhin conjectured that $\Br(X_{k^s})^{G_k}$ is finite for a smooth projective variety $X$  over a finitely generated field $k$, and they proved the finiteness of $\Br(X_{k^s})^{G_k}$ for abelian varieties and $K3$ surfaces. Our Theorem \ref{mainthm} implies the finiteness of $\Br(\CX)$ under the assumption that $\Br(X_{K^s})^{G_K}$ and $\Sha(\Pic^0_{X/K})$ are finite.

\subsection{Main results}
For two abelian groups $M$ and $N$, we say that they are \emph{almost isomorphic} if there exists a subquotient $M_1/M_0$ of $M$(resp. $N_1/N_0$ of $N$) such that $M_0$ (resp. $N_0$) and $M/M_1$ (resp. $N/N_1$) are finite and $M_1/M_0\cong N_1/N_0$. Let $f:M\rightarrow N$ be a homomorphism with a finite cokernel and a kernel almost isomorphic to an abelian group $H$. In this situation, we say that the sequence $0\rightarrow H\rightarrow M\rightarrow N\rightarrow 0$ is exact up to finite groups.
\begin{thm}\label{mainthm}
Let $\pi:\mathcal{X}\longrightarrow C$ be a proper flat morphism, where $C$ is $\Spec \CO_K$ for some number field $K$. Assume that $\mathcal{X}$ is regular and the generic fiber $X$ of $\pi$ is projective and geometrically connected over $K$. Then there is an exact sequence up to finite groups\\
$$
0\longrightarrow \Sha(\Pic_{X/K}^0)\longrightarrow \Br(\mathcal{X}) \longrightarrow  \Br(X_{\olsi{K}})^{G_K}\lra 0.
$$
\end{thm}
For arithmetic schemes of dimension $\geq 3$, the above question was first studied by Tankeev (cf.\cite{Tan1, Tan2, Tan3}), who proved the result for some special arithmetic schemes. Geisser (\cite{Gei2}) was the first to investigate fibrations of arbitrary relative dimensions over finite fields using étale motivic cohomology theory. In \cite{Qin1}, the author proved a similar result for fibrations over arbitrary finitely generated fields.
 \begin{cor}
$\Br(\CX)(\ell)$ is finite if and only if  $\Sha(\Pic^0_{X/K})(\ell)$ is finite and $T^1(X,\ell)$ holds.
\end{cor}
The statement for fibrations over finite fields was proved by Geisser \cite{Gei2}.
For a smooth projective geometrically connected variety $X$ over a finitely generated field $k$ of characteristic $0$, Orr and Skorobogatov \cite[Thm. 5.1]{OS} proved that
$\Br(X_{k^s})^{G_k}$ is finite under the assumption of the integral Mumford-Tate conjecture. Thus, Theorem \ref{mainthm} implies the following result concerning Artin's question.
\begin{cor}
Assuming the integral Mumford-Tate conjecture and the Tate-Shafarevich conjecture, then $\Br(\CX)$ is finite for all proper regular schemes over $\Spec \ZZ$.
\end{cor}

By a theorem of Andr\'e \cite{And} and a result of Ambrosi \cite[Cor. 1.6.2.1]{Amb}, the finiteness of $\Br(X_{\bar{k}})^{G_k}$ has been reduced to smooth projective surfaces over $k$. Thus, Theorem \ref{mainthm} implies the following :
\begin{thm}
 Assuming that $\Br(\CX)$ is finite for all $3$-dimensional regular proper flat schemes over $\ZZ$, then $\Br(\CX)$ is finite for all regular proper flat schemes over $\ZZ$.
\end{thm}
\begin{rem}
The finiteness of Brauer groups for smooth proper varieties over finite fields has been reduced to smooth projective surfaces  over finite fields by the work of de Jong \cite{deJ2} and Morrow \cite[Thm. 4.3]{Mor}.
\end{rem}
In the process of proving Theorem \ref{mainthm}, we also obtain some local results.
\subsubsection*{Local results}
Let $R$ denote a Henselian discrete valuation ring (DVR) of characteristic $0$ with a perfect residue field $k$ of characteristic $p>0$. Let $R^{sh}$ be a strict Henselization of $R$. Denote by $K$ (resp. $K^{sh}$) the quotient field of $R$ (resp. $R^{sh}$). Let $I$ denote $\Gal(\bar{K}/K^{sh})$ the inertia group for $K$.
\begin{prop}\label{cor1}
Let $\pi:\CX\lra S=\Spec R$ be a proper flat morphism with $\CX$ regular. Assuming that the residue field of $R$ is finite, the natural map
$$\Br(\CX_{R^{sh}})(p)\lra\Br(X_{\Bar{K}})^{I}(p)$$
has a kernel of finite exponent and a finite cokernel.
\end{prop}

\begin{rem}
In the case that $\CX$ is of dimension $2$, a theorem of Artin (\cite[Thm. 3.1]{Gro3}) implies that $\Br(\CX_{R^{sh}})=0$, and it is well-known that $\Br(X_{\Bar{K}})=0$, since $X/K$ is a smooth proper curve. Hence, the claim is trivial in this case.  For higher relative dimensions, we will use a pullback trick developed by Colliot-Thélène and Skorobogatov (cf. \cite{CTS1} or \cite{Yua} ) to show that the map has a kernel of finite exponent (which can be bounded) by reducing it to the case of relative dimension $1$. The finiteness of the cokernel follows directly from Theorem \ref{main}, which is a 
$p$-adic version of the local invariant cycle theorem. This result can be viewed as a weak generalization of Artin's theorem. An important consequence of this proposition is the following result comparing the Brauer group of the proper regular model $\CX$ to the geometrical Brauer group of $X$, which will play a key role in the proof of Theorem \ref{mainthm}.
\end{rem}
\begin{cor}\label{big}
Let $\pi:\CX\lra S=\Spec R$ be a proper flat morphism with $\CX$ regular. Assuming that the residue field of $R$ is finite, then the natural map
$$\Br(\CX)\lra\Br(X_{\Bar{K}})^{G_K}$$
has a finite kernel and a finite cokernel.
\end{cor}
\begin{rem}
Geisser-Morin \cite{GeMo} compared $\Br(\CX)$ to $\Br(Y)$, where $Y$ is the special fiber of $\pi$.
\end{rem}
Next, we derive some result on Manin's pairing from the Corollary \ref{big}.
Let $X$ be  a proper smooth geometrically connected variety over a $p$-adic local field $K$. Let $\cH_0(X)$ be the Chow group of $0$-cycles on $X$ modulo rational equivalence. Manin \cite{Man} defined a natural pairing 
$$\cH_0(X)\times \Br(X) \lra \Br(K)\cong \QQ/\ZZ.$$
Let $A_0(X)\subset \cH_0(X)$ be the subgroup consisting of cycles of degree $0$. Then there is a homomorphism  
$$ \phi_X: A_0(X) \lra \Alb_X(K) $$
induced by the Albanese map $X\lra \Alb_X$, where $\Alb_X$ denotes the Albanese variety of $X$ and $\Alb_X(K)$ denotes the group of $K$-rational points on $\Alb_X$. Colliot-Thélène \cite{CT} conjectured that the kernel of $\phi_X$ is the direct sum of a finite group and a divisible group. Manin's pairing induces a natural map
$$ A_0(X)\lra \Hom(\Br(X)/\Br(K), \QQ/\ZZ) $$
which was shown to be an isomorphism when $\dim(X)=1$ by Lichtenbaum\cite{Lic}. Assuming that $X$ has a regular model $\CX$ which is proper flat of finite type over the integer ring $R$ of $K$, the natural map above induces a homomorphism 
$$ A_0(X)\lra \Hom(\Br(X)/\Br(K)+\Br(\CX), \QQ/\ZZ) $$
which was shown to be surjective by Saito and Sato \cite[Thm. 1.1.3]{SS1}. The Hochschild–Serre spectral sequence 
$$E_2^{p,q}=H^p(G_K, H^q(X_{\bar{K}},\GG_m))\Rightarrow H^{p+q}(X,\GG_m)$$
gives a long exact sequence
$$ \Br(K)\lra \Ker(\Br(X)\lra\Br(X_{\bar{K}})^{G_K})\lra H^1(K,\Pic(X_{\bar{K}})) \lra H^3(K, \GG_m).$$
By \cite[Chap. II, Prop. 1.5]{Mil1}, $H^3(K,\GG_m)=0$, thus it induces a natural map
$$ H^1(K, \Pic^0_{X/K}) \lra \Br(X)/\Br(K).$$
By taking dual and composing with the natural isomorphism $\Hom(H^1(K,\Pic^0_{X/K}), \QQ/\ZZ)\cong \Alb_X(K)$ induced by the local Tate duality \cite[Chap. I, Cor. 3.4]{Mil1}, we get a natural homomorphism
$$ \psi_X : \Hom(\Br(X)/\Br(K)+\Br(\CX), \QQ/\ZZ) \lra \Alb_X(K).$$
As a result of Corollary \ref{big}, $\psi_X$ has a finite kernel and a finite cokernel.
\begin{thm}\label{0cycle}
There is a commutative diagram
\begin{displaymath}
\xymatrix{ A_0(X)\ar[r] \ar[rd]^{\phi_X} & \Hom(\Br(X)/\Br(K)+\Br(\CX), \QQ/\ZZ) \ar[d]^{\psi_X}\\
& \Alb_X(K)}
\end{displaymath}
and $\psi_X$ has a finite kernel and a finite cokernel. Let $f$ (resp. $g$) denote the natural map $\Br(\CX)\rightarrow \Br(X_{\bar{K}})^{G_K}$ (resp. $\Br(X)\rightarrow \Br(X_{\bar{K}})^{G_K}$) and $F$ denote the cokernel of $\Pic(X_{\bar{K}})^{G_K}\rightarrow \NS(X_{\bar{K}})^{G_K}$. Then there is a long exact sequence of finite groups
$$0\lra F\lra \Coker(\psi_X)^\vee  \lra \Ker(f)/\Br(\CX)\cap \Br(K) \lra H^1(K, \NS(X_{\bar{K}}))\\$$
$$\lra \Ker(\psi_{X})^\vee \lra \Coker(f)\lra \Coker(g) \lra 0.$$
\end{thm}
\begin{rem}\label{rem11}
The horizontal map in the diagram was proved to be surjective by Saito and Sato \cite{SS1}. Thus, $\Coker(\psi_X)=\Coker(\phi_X)$. Kai \cite{Kai} proved the commutativity of the diagram and showed $F\cong \Coker(\psi_X)^\vee$ for specific varieties $X$ satisfying certain conditions. Furthermore, it was shown by Saito and Sato \cite [Cor. 0.4]{SS2} that $A_0(X)$ is $\ell$-divisible for almost all primes $\ell$. As a consequence, $\Hom(\Br(X)/\Br(K)+\Br(\CX), \QQ/\ZZ)$ is $\ell$-divisible for almost all $\ell$. Therefore, the pro-$\ell$-part of $\Hom(\Br(X)/\Br(K)+\Br(\CX), \QQ/\ZZ)$ is actually trivial for almost all $\ell$. By \cite[Cor. 2.6]{CTSa} or \cite[Thm. 0.3]{SS2},  the prime-to-$p$-part of it is finite. Using a similar argument, the finiteness of $\Ker(\psi_X)$ can indeed be derived from a conjecture of Colliot-Thélène's \cite{CT}, namely that $\Ker(\phi_X)$ is the direct sum of a finite group and a divisible group.
\end{rem}

\subsection{Sketch of Proofs }
\subsubsection{Sketch of the proof of Proposition \ref{cor1}}
Proposition \ref{cor1} is a consequence of the following $p$-adic version of the local invariant cycle theorem: the natural map
$$H^2_{\syn}(\mathcal{X}_{R^{sh}},\mathbb{Q}_p(1))\lra H^2_{\et}(X_{\bar{K}},\QQ_p(1))^{I}$$
is surjective. The idea of proving it is to use the semi-stable $p$-adic Hodge comparsion theorem, syntomic cohomology for $\CX$ and weight arguments in Deligne's proof of the local invariant cycle theorem in equal characteristic (cf. \cite[Thm. 3.6.1]{Del2}).

K. Sato defined a $p$-adic etale Tate twist $\mathfrak{T}_n(1)\in D^b(\mathcal{X}_{\et},\mathbb{Z}/p^n\mathbb{Z})$ in \cite[\S1.3]{Sat}, which is naturally isomorphic to $\mathbb{G}_m\otimes^\LL\mathbb{Z}/p^n\mathbb{Z}[-1]$. Thus, there is a canonical isomorphism
 $$ H^2_{\fppf}(\CX_{R^{sh}}, \mu_{p^n})\cong H^2_{\et}(\CX_{R^{sh}},\mathfrak{T}_n(1)).
 $$
 The $p$-adic etale Tate twist $\mathfrak{T}_n(1)$ is related to the truncated nearby cycle by the following distinguished triangle
$$
i_*\nu_{Y,n}^{0}[-2]\longrightarrow \mathfrak{T}_n(1)\longrightarrow \tau_{\leq 1}Rj_*\mathbb{Z}/p^n(1) \longrightarrow
i_*\nu_{Y,n}^{0}[-1],
$$
where $i:Y_{\bar{k}}\lra \CX_{R^{sh}}$ is the special fiber and $\nu_{Y,n}^{0}$ is roughly equal to the constant etale sheaf $\ZZ/p^n\ZZ$ on $Y_{\bar{k}}$ ( equal to $\ZZ/p^n\ZZ$ if $Y$ is smooth). Taking cohomology, we get an exact sequence
$$\varprojlim_n H^2_{\et}(\mathcal{X}_{R^{sh}},\mathfrak{T}_n(1))\otimes_\ZZ \mathbb{Q} \rightarrow \varprojlim_n H^2_{\et}(\mathcal{X}_{R^{sh}},\tau_{\leq 1}Rj_*\mathbb{Z}/p^n(1))\otimes_\ZZ \mathbb{Q}\rightarrow 
\varprojlim_n H^1_{\et}(Y_{\bar{k}},\nu_{Y,n}^0)\otimes_\ZZ \mathbb{Q},
$$
where the third term is roughly equal to $H^1_{\et}(Y_{\bar{k}},\QQ_p)$ which is of weight $\geq 1$. By using the semi-stable p-adic Hodge comparsion theorem, we can show that
$$H^2_{\et}(X_{\bar{K}},\QQ_p(1))^{I}$$
is of weight $\leq 0$. So the question is reduced to show that the natural map
$$\varprojlim_n H^2_{\et}(\mathcal{X}_{R^{sh}},\tau_{\leq 1}Rj_*\mathbb{Z}/p^n(1))\otimes_\ZZ \mathbb{Q}\lra H^2_{\et}(X_{\bar{K}},\QQ_p(1))^{I}$$
is surjective. By a result of Kato and others, there is a canonical isomorphism
$$\mathscr{S}_n(1)_{\mathcal{X}}\cong\tau _{\leq 1}i^*Rj_*\mathbb{Z}/p^n(1),$$
where $\mathscr{S}_n(1)_{\mathcal{X}}$ is the (log) syntomic sheaf modulo $p^n$ on $Y_{\et}$. Thus, the question is reduced to show that the natural map
$$H^2_{\syn}(\mathcal{X}_{R^{sh}},\mathbb{Q}_p(1))\lra H^2_{\et}(X_{\bar{K}},\QQ_p(1))^{I}$$
is surjective. This follows from the degeneracy at the $E_2$ page of a syntomic descent spectral sequence(cf.\cite[Thm. A(5)]{NN} and \cite{DN})).

\subsubsection{The idea of the proof of Theorem \ref{mainthm}}
The left exactness of the sequences can be proved by combining Grothendieck's argument in \cite[\S 4]{Gro3} with a pullback technique developed by Colliot-Thélène and Skorobogatov in \cite{CTS1}. The tricky part is the  finiteness of the cokernel of the natural map
$$\Br(\mathcal{X}) \longrightarrow  \Br(X_{\olsi{K}})^{G_K}.$$
Colliot-Thélène \cite[Conj. 1.5(d)]{CT} conjectured that the cokernel of the natural map 
$$\prod_v A_0(X_{K_v})\lra \Hom(\Br(X)/\Br(K),\QQ/\ZZ )$$
induced by the Brauer-Manin pairing, is a finite group that can be identified with the group
$$\Hom(\Br^\prime(\CX)/\Br(\CO_K),\QQ/\ZZ),$$
where $\Br^\prime(\CX)\subset\Br(\CX)$ is the subgroup of elements vanishing on $X(K_v)$ for all real places of $K$. The dual statement of the conjecture predicts that the natural map 
$$\Br(X)/(\Br(\CX)+\Br(K)) \lra \prod_{v\in C^{\circ}}\Br(X_{K_v})/(\Br(\CX_{\CO_v})+\Br(K_v))$$
has a finite kernel, where $C^{\circ}$ denotes the set of finite places. In conjuction with a theorem of Colliot-Thélène and Skorobogatov \cite{CTS1}, which states that the natural map $\Br(X)\lra\Br(X_{\bar{K}})^{G_K}$ has a finite cokernel, it indicates that the natural map
$$\Br(X_{\olsi{K}})^{G_K}/\Br(\CX) \lra \prod_{v} \Br(X_{\olsi{K_v}})^{G_{K_v}}/\Br(\CX_{\CO_v})$$
has a finite kernel. We will prove this by using the pullback trick of Colliot-Thélène and Skorobogatov. By Corollary \ref{big}, the group $\Br(X_{\olsi{K_v}})^{G_{K_v}}/\Br(\CX_{\CO_v})$ is finite for all finite places $v$. This will imply that the divisible part of $\Br(X_{\olsi{K}})^{G_K}/\Br(\CX)$ is trivial. Consequently, $(\Br(X_{\olsi{K}})^{G_K}/\Br(\CX))(\ell)$ is finite for all prime $\ell$. By using the Fontaine-Laffaille-Messing theory( cf. \cite[\S 3]{BM} ), we will show that for sufficiently large $\ell$, the group $(\Br(X_{\olsi{K_v}})^{G_{K_v}}/\Br(\CX_{\CO_v}))(\ell)$ is trivial for all $v$ such that $\CX_{\CO_v}$ is smooth over $\CO_v$. This will imply that $(\Br(X_{\olsi{K}})^{G_K}/\Br(\CX))(\ell)$ is trivial for $\ell\gg0$.

\subsection{Notation and Terminology}
\subsubsection*{Fields}
By a \emph{finitely generated field}, we mean a field which is finitely generated over a prime field.\\
For any field $k$, denote by $k^s$ the separable closure. Denote by $G_k=\Gal(k^s/k)$ the absolute Galois group of $k$.
\subsubsection*{Henselization}
Let $R$ be a noetherian local ring, denote by $R^h$ (resp. $R^{sh}$) the Henselization (resp. strict Henselization) of $R$ at the maximal ideal. If $R$
is a domain, denote by $K^h$ (resp. $K^{sh}$) the fraction field of $R^{h}$ (resp. $R^{sh}$). 
\subsubsection*{Varieties}
By a \emph{variety} over a field $k$, we mean a scheme which is separated and of finite type over $k$. For a smooth proper geometrically connected variety $X$ over a field $k$, we use $\Pic^0_{X/k}$ to denote the identity component of the Picard scheme $\Pic_{X/k}$.
\subsubsection*{Cohomology}
The default sheaves and cohomology over schemes are with respect to the
small \'etale site. So $H^i$ is the abbreviation of $H_{\et}^i$. Denote by $H^i_{\fppf}$ the cohomology with respect to the fppf site.
\subsubsection*{Brauer groups}
For any noetherian scheme $X$, denote the \emph{cohomological Brauer group}
$$
\Br(X):=H^2(X,\mathbb{G}_m)_{\tor}.
$$
\subsubsection*{Tate-Shafarevich group}
For an abelian variety $A$ defined over a global field $K$, denote the  \emph{Tate–Shafarevich group}

$$
\Sha(A):=\Ker(H^1(K, A)\longrightarrow \prod _v H^1(K_v, A)),
$$
where $K_v$ denotes the completion of $K$ at a place $v$. Note that if $v$ is a finite place, the natural map $H^1(K_v^h, A) \rightarrow H^1(K,A)$
is an isomorphism (cf. \cite[Chap. II, Rem. 3.10]{Mil1}), where $K_v^h$ denotes the fraction field of the Henselian local ring at $v$.
\subsubsection*{Abelian group}
For any abelian group $M$, integer $m$ and prime $\ell$, we set\\
$$M[m]=\{x\in M| mx=0\},\quad M_{\tor}=\bigcup\limits_{m\geq 1}M[m],\quad  M(\ell)=\bigcup\limits_{n\geq 1}M[\ell^n], $$
$$ M(\non \ell)=\bigcup\limits_{m\geq 1, \ell\nmid m}M[m],\ T_\ell M=\Hom_\mathbb{Z}(\mathbb{Q}_\ell/\mathbb{Z}_\ell, M)=\lim \limits_{\substack{\leftarrow \\ n}}M[\ell^n],\ V_\ell M= T_\ell(M)\otimes_{\mathbb{Z}_\ell}\mathbb{Q}_\ell.$$
A torsion abelian group $M$ is \emph{of cofinite type} if $M[\ell]$ is finite for all prime $\ell$.


\bigskip
\bigskip
\noindent {\bf Acknowledgements}:

\noindent I would like to thank my advisor Xinyi Yuan for suggesting this question and introducing me to the pullback trick.  I also wish to thank Ruochuan Liu, Wiesława Nizioł, Thomas Geisser, Veronika Ertl and Zhenghui Li for helpful suggestions and comments.

\section{Preliminary}
In this section, we will review some well-established facts that have either appeared in the literature or are familiar to experts.
\subsection*{Basic exact sequences}
\begin{prop}\label{prop2.1}
Let $X$ be a smooth projective geometrically connected variety over a field $k$. Define $\NS(X)$ to be the image of the map $\Pic(X)\rightarrow \NS(X_{k^s})$. Let $\ell$ be a prime different from char$(k)$, then we have
\begin{itemize}
\item[(a)]
For any $\ell\neq \Char(k)$, the exact sequence of $G_k$-representations
$$
0\longrightarrow \NS(X_{k^s})\otimes_\ZZ\mathbb{Q}_\ell\longrightarrow H^2(X_{k^s},\mathbb{Q}_\ell(1)) \longrightarrow V_\ell \Br(X_{k^s})\longrightarrow 0
$$
is split. Taking $G_k$-invariant, there is an exact sequence
$$
0\longrightarrow \NS(X)\otimes_\ZZ\mathbb{Q}_\ell\longrightarrow H^2(X_{k^s},\mathbb{Q}_\ell(1))^{G_k} \longrightarrow V_\ell \Br(X_{k^s})^{G_k}\longrightarrow 0.
$$
\item[(b)]
For sufficiently large $\ell$, the exact sequence of $G_k$-modules
$$
0\longrightarrow \NS(X_{k^s})/\ell^n\longrightarrow H^2(X_{k^s},\ZZ/\ell^n(1))\longrightarrow \Br(X_{k^s})[\ell^n]\longrightarrow 0
$$ 
is split for any $n\geq 1$. Moreover, for all but finitely many $\ell$, there is an exact sequence
$$
0\longrightarrow \NS(X)\otimes_\ZZ\mathbb{Q}_\ell/\ZZ_\ell\longrightarrow H^2(X_{k^s},\QQ_\ell/\ZZ_\ell(1))^{G_k} \longrightarrow \Br(X_{k^s})^{G_k}(\ell)\longrightarrow 0.
$$
\end{itemize}
\end{prop}
\begin{proof}
We refer the proof of $(a)$ and the first claim of $(b)$ to \cite[Prop. 2.1]{Qin2}. Taking $G_k$-invariants, we get an exact sequence for all but finitely many $\ell$
$$
0\longrightarrow (\NS(X_{k^s})/\ell^n)^{G_k}\longrightarrow H^2(X_{k^s},\ZZ/\ell^n(1))^{G_k}\longrightarrow \Br(X_{k^s})^{G_k}[\ell^n]\longrightarrow 0.
$$ 
Taking direct limit, we get
$$
0\longrightarrow (\NS(X_{k^s})\otimes_{\ZZ}\mathbb{Q}_\ell/\ZZ_\ell)^{G_k}\longrightarrow H^2(X_{k^s},\QQ_\ell/\ZZ_\ell(1))^{G_k} \longrightarrow \Br(X_{k^s})^{G_k}(\ell)\longrightarrow 0.
$$
To prove the second claim of $(b)$, it suffices to show that the natural map
$$\NS(X)\otimes_\ZZ\mathbb{Q}_\ell/\ZZ_\ell\lra (\NS(X_{k^s})\otimes_\ZZ \QQ_\ell /\ZZ_\ell)^{G_k}$$
is an isomorphism for all but finitely many $\ell$. Consider the exact sequence
$$0\lra\NS(X)\otimes_{\ZZ}\ZZ_\ell \lra \NS(X_{k^s})\otimes _\ZZ\ZZ_\ell\lra \NS(X_{k^s})/\NS(X)\otimes_{\ZZ}\ZZ_\ell\lra 0.$$
For $\ell$ sufficiently large, the last group will be a free $\ZZ_\ell$-module. Thus, it is split for all but finitely many $\ell$. So tensoring $\QQ_\ell/\ZZ_\ell$, the sequence is still exact. This proves injectivity.
Since the $G_k$ action on $\NS(X_{k^s})$ factors through a finite group $G$, thus, for $\ell$ prime to $|G|$, any element in the $(\NS(X_{k^s})\otimes_\ZZ \QQ_\ell /\ZZ_\ell)^{G_k}$ can be written as $\sum_{g\in G} gx$. This implies that the natural map
$$\NS(X_{k^s})^{G_k}\otimes_\ZZ \QQ_\ell /\ZZ_\ell\lra (\NS(X_{k^s})\otimes_\ZZ \QQ_\ell /\ZZ_\ell)^{G_k}$$
is surjective. Thus, to prove the surjectivity, it suffices to show that
$$ \NS(X)\otimes_\ZZ \QQ_\ell \lra \NS(X_{k^s})^{G_k}\otimes_\ZZ \QQ_\ell $$
is surjective. This was proved in \cite[\S2.2]{Yua}.
\end{proof}
\subsection*{Cofiniteness}
\begin{prop}\label{cofin}
Let $\CX$ be an integral regular scheme flat and of finite type over $\Spec(\ZZ)$. Then $\Br(\CX)$ is of cofinite type. 
\end{prop}
\begin{proof}
 Let $\ell$ be a prime. By shrinking $\CX$, we may assume that $\CX$ is a $\ZZ[1/\ell]$-scheme. Let $\pi:\CX \lra S=\Spec(\ZZ[1/\ell])$ the structure morphism.
 By the Kummer exact sequence, there is a surjective map
 $$H^2(\CX, \mu_\ell)\lra \Br(\CX)[\ell].$$
 Thus, it suffices to show that $H^2(\CX, \mu_\ell)$ is finite. By the Leray spectral sequence
 $$H^p(S,R^q\pi_*\mu_\ell)\Rightarrow H^{p+q}(\CX,\mu_\ell),$$
 it suffices to show that $H^p(S,R^q\pi_*\mu_\ell)$ is finite. By \cite[Thm. 9.5.1]{Fu}, $R^q\pi_*\mu_\ell$ is a constructible $\ell$-torsion sheaf.
 By \cite[ Chap. II, Thm. 2.13]{Mil1}, $H^p(S,R^q\pi_*\mu_\ell)$ is finite.
\end{proof}
\begin{prop}[Prop. 5.4 in \cite{Gei3}]\label{twosha}
Let $X$ be a smooth projective geometrically connected variety over a number field $K$. Define
$$\Sha(\Pic_{X/K}):=\Ker(H^1(K, \Pic_{X/K}) \lra \prod_{v}H^1(K_v^h, \Pic_{X/K})).$$
Then the natural map
$$\Sha(\Pic^0_{X/K})\lra \Sha(\Pic_{X/K})$$
has a finite kernel and a cokernel of finite exponent.
\end{prop}
\subsection*{Colliot-Th\'el\`ene and Skorobogatov's pullback trick}
In this section, we will recall a pullback trick developed by Colliot-Th\'el\`ene and Skorobogatov in their paper \cite{CTS1}. This pullback trick plays an essential role in the proof of our main theorems.
\begin{lem}\label{thekeytrick}
Let $X$ be a smooth projective geometrically connected variety over a field $k$. 
\begin{itemize}
    \item[(i)] There exists a finite separable extension $k^\prime/k$ and smooth projective geometrically connected curves $C_1,...,C_m\subset X_{k^\prime}$ and an abelian subvariety $B\subset\prod_{i=1}^{m}\Pic^0_{C_i/k^\prime}$ such that the induced morphism of $k^s$-points
$$
\Pic(X_{k^s})\times B(k^s) \lra \bigoplus_{i=1}^m \Pic(C_{i,k^s})
$$
has a kernel and a cokernel of finite exponent.
\item[(ii)] There exist smooth projective integral curves $C_1,...,C_m\subset X$ over $k$ and a continuous $G_k$-module $B$ with a $G_k$-equivariant map $B\ra\oplus_{i=1}^{m}\Pic(C_{i,k^s})$ such that the induced $G_k$-morphism 
$$
\Pic(X_{k^s})\times B \lra \bigoplus_{i=1}^m \Pic(C_{i,k^s})
$$
has a kernel and a cokernel of finite exponent.
\end{itemize}
\end{lem}
\begin{proof}
For the proof of (i), see \cite[p. 17]{Yua}. Let $C_1,...,C_m$ be curves as in (i). As a $k$-scheme, $C_i$ is a smooth projective integral curve over $k$. The induced morphism
$$
\Pic^0_{X/k,\red}\lra\prod_{i=1}^{m}\Pic^0_{C_i/k}
$$
has a finite kernel. By \cite[p. 17]{Yua}, there exists a morphism of abelian varieties $B^\prime \lra\prod_{i=1}^{m}\Pic^0_{C_i/k}$ such
that the induced morphism
$$\Pic^0_{X/k,\red}\times B^\prime \lra \prod_{i=1}^{m}\Pic^0_{C_i/k}$$
is an isogeny. It follows that the natural map
$$
\Pic(X_{k^s})\times B^\prime(k^s) \lra \bigoplus_{i=1}^m \Pic(C_{i,k^s})
$$
has a kernel of finite exponent and a cokernel $M$ of finite dimension after tensoring with $\QQ$. There exists a $G_k$-submodule $N\subset \bigoplus_{i=1}^m\Pic(C_{i,k^s})$ such that $N_\QQ\stackrel{\thicksim}{\ra}M_\QQ$. We may assume that $N$ is finitely generated over $\ZZ$. It suffices to show that 
$$\Pic(X_{k^s})\oplus B^\prime(k^s) \oplus N \lra \bigoplus_{i=1}^m \Pic(C_{i,k^s})
$$
has a kernel and a cokernel of finite exponent. It is easy to see that the induced map
$$\NS(X_{k^s})_\QQ\times N_\QQ \lra (\bigoplus_{i=1}^m \NS(C_{i,k^s}))_\QQ$$
is surjective. So 
$$\NS(X_{k^s})\times N \lra \bigoplus_{i=1}^m \NS(C_{i,k^s})$$
has a finite cokernel. This implies that
$$\Pic(X_{k^s})\oplus B^\prime(k^s) \oplus N \lra \bigoplus_{i=1}^m \Pic(C_{i,k^s})
$$
has a cokernel of finite exponent. Let $(a,b,c)$ be an element in the kernel. Since $N_\QQ\stackrel{\thicksim}{\ra}M_\QQ$, $c\in N_\tor$. Since $|N_\tor|$ is finite, $|N_\tor|(a,b,c)$ lies in the kernel of 
$$
\Pic(X_{k^s})\times B^\prime(k^s) \lra \bigoplus_{i=1}^m \Pic(C_{i,k^s}),
$$
which is of finite exponent.
\end{proof}
\section{Local results}
\subsection{Local invariant cycle theorems}
\begin{conj}\label{conj0}
Let $\pi:\CX\lra S=\Spec R$ be a proper flat morphism with $\CX$ regular. Assume that the residue field of $R$ is finite of characteristic $p>0$. Let $X$ (resp. $Y$ ) denote the generic (resp. special) fiber of $\pi$ and $\ell \neq p$ be a prime. Then the natural map
$$H^i_{\et}(Y_{\bar{k}}, \QQ_\ell) \lra H^i_{\et}(X_{\bar{K}}, \QQ_\ell)^{I} $$
is surjective for all $i\geq 0$.
\end{conj}
The conjecture is known as \emph{the local invariant cycle conjecture} in mixed characteristic and was established for $i\leq 2$ as a consequence of the weight monodromy conjecture, which was proven for $\dim X\leq 2$ (cf. \cite{RZ} and \cite[\S 10]{FlMo}). By the proper base change theorem, $H^i_{\et}(Y_{\bar{k}}, \QQ_\ell)$ can be identified with $H^i_{\et}(\CX_{R^{sh}}, \QQ_\ell)$, the conjecture is equivalent to the surjectivity of 
$$H^i_{\et}(\CX_{R^{sh}}, \QQ_\ell) \lra H^i_{\et}(X_{\bar{K}}, \QQ_\ell)^{I}.$$
In this section, we will prove a $p$-adic analogue of the above conjecture for $i=2$ and deduce Proposition \ref{cor1} from this result.
\begin{thm}\label{main}
Let $\pi:\CX\lra S=\Spec R$ be a proper flat morphism with $\CX$ regular. Assume that the residue field of $R$ is finite of characteristic $p>0$. Let $X$ denote the generic fiber of $\pi$. Define:
$$ H^2_{\fppf}(\CX_{R^{sh}}, \QQ_p(1)):=\varprojlim_{n} H^2_{\fppf}(\CX_{R^{sh}},\mu_{p^n})\otimes_{\ZZ_p}\QQ_p,$$ 
then the natural map
$$ H^2_{\fppf}(\CX_{R^{sh}}, \QQ_p(1)) \lra H^2_{\et}(X_{\bar{K}},\QQ_p(1))^{I}
$$
is surjective.
\end{thm}
\begin{rem}
In the case that $\pi$ is a strictly semi-stable morphism, assuming the conjecture of purity of the weight filtration( cf. \cite[Conj. 2.6.5]{Illu}),  by the same method, one can show that the claim holds for any $i\geq 0$.

\end{rem}
In the case that $\pi$ is smooth and proper, by the same argument as in the proof of Theorem \ref{main}, one can prove the following result:

\begin{prop}
Let $\pi:\CX\lra S=\Spec R$ be a smooth and proper morphism. Assume that 
the residue field of $R$ is finite of characteristic $p>0$. Let $X$ denote the generic fiber of $\pi$. Define:
$$ H^i_{L}(\CX_{R^{sh}}, \QQ_p(m)):=\varprojlim_{n} \HH^2_{\et}(\CX_{R^{sh}},\ZZ(m)_{\et}\otimes^{\LL}\ZZ/p^n)\otimes_{\ZZ_p}\QQ_p,$$ 
where $\ZZ(m)_{\et}$ denotes Bloch's cycle complex for the etale topology on $\CX_{R^{sh}}$. Then there is a surjective natural map
$$H^i_{L}(\CX_{R^{sh}}, \QQ_p(m)) \lra H^i_{\et}(X_{\bar{K}},\QQ_p(m))^{I}.
$$
\end{prop}
When $m=0$, it was shown in \cite[I. (4.1)]{FM} that the natural map $H^i_{\et}(\CX_{R^{sh}}, \QQ_p) \rightarrow H^i_{\et}(X_{\bar{K}},\QQ_p)^{I}$ is an isomorphism as a consequence of Fontaine’s $C_{crys}$-comparison isomorphism in $p$-adic Hodge theory, proven by Fontaine and Messing \cite{FM} and Faltings \cite{Fal}.

In the case that $\pi$ is a semi-stable morphism, Sato \cite[\S1.3]{Sat} defined a $p$-adic etale Tate twist $\mathfrak{T}_n(m)\in D^b(\mathcal{X}_{\et},\mathbb{Z}/p^n\mathbb{Z})$ and conjectured that $\mathfrak{T}_n(m)$ is naturally isomorphic to $\ZZ(m)_{\et}\otimes^{\LL}\ZZ/p^n$, which was proved by Geisser when $\pi$ is smooth and by Sato \cite[Prop. 4.5.1]{Sat} when $m\leq 1$. Assuming that the inverse system of $H^i(Y_{\bar{k}}, \nu_{Y,n}^{m}), n\geq 1$ satisfies the Mittag-Leffler condition (it is known to hold when $Y$ is smooth and proper or $m=0$), where $\nu_{Y,n}^{m}$ are the logarithmic Hodge–Witt sheaves defined by Sato in \cite{Sat1}, which agree with the logarithmic de Rham-Witt sheaves $W_n\Omega_{Y, \log}^m$ when $Y$ is smooth, one can prove that the following natural map is surjective under the assumption of the conjecture of purity of the weight filtration( cf. \cite[Conj. 2.6.5]{Illu})
$$\varprojlim_n \HH^i_{\et}(\mathcal{X}_{R^{sh}},\mathfrak{T}_n(m))\otimes_{\ZZ_p} \mathbb{Q}_p\lra H^i_{\et}(X_{\bar{K}},\QQ_p(m))^{I}. $$
This leads to the following conjecture.
\begin{conj}
Let $\pi:\CX\lra S=\Spec R$ be a proper flat morphism with $\CX$ regular. Assume that 
the residue field of $R$ is finite. Let $X$ denote the generic fiber of $\pi$, then there is a surjective natural map
$$H^i_{L}(\CX_{R^{sh}}, \QQ_\ell(m)) \lra H^i_{\et}(X_{\bar{K}},\QQ_\ell(m))^{I}.
$$
\end{conj}
\begin{rem}
In the case that $\ell$ is invertible in $R$, $\ZZ(m)_{\et}\otimes^{\LL}\ZZ/\ell^n$ was conjectured to be isomorphic to the etale sheaf $\mu_n^{m}$ (cf. \cite[Thm. 12.5]{Lev}), assuming this, then the conjecture is exactly same as Conjecture \ref{conj0}.
\end{rem}

\subsection{Proof of Theorem \ref{main}}
We will prove Theorem \ref{main} under the assumption that $\pi$ is a semi-stable projective morphism. The general case can be reduced to this case by de Jong's alteration theorem \cite[Cor. 5.1]{deJ1}( cf. Theorem \ref{thm25} and Remark \ref{rmkforthm3.2} below).

\begin{defn}
Let $\pi:\CX\lra S=\Spec R$ be a flat separated morphism of finite type. Assume that $\CX$ is regular and irreducible. Let $Y$ be the special fiber of $\pi$. Let $Y_i, i\in I$ be the irreducible components of $Y$. Put $Y_J =\cap_{j\in J} Y_j$ (scheme-theoretic intersection) for a non empty subset $J$ of $I$. We say $\CX$ is \emph{strictly semi-stable} over $S$ if the following property hold:
\begin{itemize}
\item[a)] the generic fiber $X$ of $\pi$ is smooth over $K=\mathrm{Frac}(R)$,
\item[b)] $Y$ is a reduced scheme,
\item[c)] and for each nonempty $J\subseteq I$, $Y_J$ is smooth over $k$ and has codimension $\#J$ in $\CX$.
\end{itemize}
We say $\CX$ is \emph{semi-stable} over $S$ (or with a \emph{semi-stable reduction}) if the situation etale locally looks as described above.

\end{defn}

\begin{prop}\label{thm13}
Let $\pi:\mathcal{X}\longrightarrow S=\Spec R$ be a semi-stable flat projective morphism, where $R$ is a Henselian DVR of characteristic $0$ with a perfect residue field $k$ of characteristic $p>0$. Let $i:Y\longrightarrow \mathcal{X}$ be the special fiber and $j:U=X_K\longrightarrow \mathcal{X}$ be the generic fiber. Then the natural map 
$$ \tau_{\leq 1}Rj_*\mathbb{Z}/p^n(1)\longrightarrow Rj_*\mathbb{Z}/p^n(1) $$
induces a surjection:
$$\varprojlim_n H^2_{\et}(\mathcal{X},\tau_{\leq 1}Rj_*\mathbb{Z}/p^n(1))\otimes_{\ZZ}\mathbb{Q}\longrightarrow H^2_{\et}(X_{\bar{K}},\mathbb{Q}_p(1))^{G_K}.$$
\end{prop}
To prove the above statement, we need the following two lemmas:
\begin{lem}[Corollary 5.2 (ii) and Theorem 5.4 in \cite{CN}]
\label{lem14}
Notations as above. Assume that $R$ is complete. Equips $\mathcal{X} $  with the log-structure defined by the special fiber. 
Let $\mathscr{S}_n(r)_\mathcal{X}$ be the (log) syntomic sheaf modulo $p^n$ on $Y_{\et}$, there exist period morphisms\\
$$\alpha_{r,n}^{FM}:\mathscr{S}_n(r)_{\mathcal{X}}\longrightarrow i^*Rj_*\mathbb{Z}/p^n(r)_{X_K}^\prime $$ 
from logarithmic syntomic cohomology to logarithmic p-adic nearby cycles. Here we set $\mathbb{Z}/p^n(r)^\prime :=\frac{1}{p^{a(r)}}\mathbb{Z}_p(r)\otimes _\ZZ\mathbb{Z}/p^n$ where $a(r)=[r/(p-1)]$.
Then we have 
\begin{itemize}
\item[(i)]
 for $i\geq r+1$, $\mathscr{H}^i(\mathscr{S}_n(r))$ is annihilated by $p^{N(r)}$.
 \item[(ii)] 
 for $0\leq i\leq r$, the kernel and cokernel of the period map
$$\alpha_{r,n}^{FM}:\mathscr{H}^i(\mathscr{S}_n(r)_{\mathcal{X}})\longrightarrow i^*R^ij_*\mathbb{Z}/p^n(r)_{X_K}^\prime $$ 
 is annihilated by $p^N$ for
an integer $N = N(K, p, r)$, which only depends on $K$, $p$, $r$.
\end{itemize}
\end{lem}
\begin{rem}
For $0\leq r\leq p-2$ and $n\geq 1$, it is known that
$$
\alpha_{r,n}^{FM}:\mathscr{S}_n(r)_{\mathcal{X}}\longrightarrow \tau _{\leq r}i^*Rj_*\mathbb{Z}/p^n(r)_{X_K} $$ 
is an isomorphism for $\CX$ a log-scheme log-smooth over a Henselian discrete
valuation ring $R$ of mixed characteristic. This was proved by Kato \cite{Ka1}\cite{Ka2}, Kurihara \cite{Kur}, and Tsuji \cite{Ts1}\cite{Ts2}. If $p>2$, taking $r=1$, their results are sufficient for our purposes. To address the case  $p=2$, we require the general result from the above result of Colmez and Niziol \cite{CN} .
\end{rem}

\begin{lem} [Thm. A(5) in \cite{NN}, \cite{DN}]\label{lem15}
Notations as above, define $R\Gamma_{\syn}(\mathcal{X},r)_n := R\Gamma(\mathcal{X}_{\et},\mathscr{S}_n(r))$ and  $R\Gamma_{\syn}(\mathcal{X},r):=\mathrm{homlim}_n R\Gamma_{\syn}(\mathcal{X},r)_n$. 
 Define $H^i_{\syn}(\mathcal{X},\mathbb{Q}_p(r)):=H^i(R\Gamma_{\syn}(\mathcal{X},r))\otimes \mathbb{Q}$. Then there is a spectral sequence 
 $$
\sideset{^{\syn}}{^{i,j}_2}{\mathop{\mathrm{E}}}=H^i_{st}(G_K,H^j_{\et}(X_{\bar{K}},\mathbb{Q}_p(r)))\Rightarrow H^{i+j}_{\syn}(\mathcal{X},\mathbb{Q}_p(r)),
$$
which is compatible with the Hochschild-Serre spectral sequence for etale cohomology
$$
\sideset{^{\et}}{^{i,j}_2}{\mathop{\mathrm{E}}}=H^i(G_K,H^j_{\et}(X_{\bar{K}},\mathbb{Q}_p(r)))\Rightarrow H^{i+j}_{\et}(X,\mathbb{Q}_p(r)).
$$
It degenerates at $E_2$ for $X_K$ projective and smooth. In particular, the edge map 
 $$H^2_{\syn}(\mathcal{X},\mathbb{Q}_p(1))\longrightarrow H^2_{\et}(X_{\bar{K}},\mathbb{Q}_p(1))^{G_K}$$
is surjective.
\end{lem}

Proof of Proposition $\ref{thm13}$.
\begin{proof}
Replacing $R$ by its completion, we can assume that $R$ is complete.
By Lemma \ref{lem14}, the maps
$$
\mathscr{H}^i(\tau_{\leq 1}\mathscr{S}_n(1)_{\mathcal{X}})\longrightarrow \mathscr{H}^i(\mathscr{S}_n(1)_{\mathcal{X}}) $$
and 
$$\mathscr{H}^i(\tau_{\leq 1}\mathscr{S}_n(1)_{\mathcal{X}})\longrightarrow \mathscr{H}^i(i^*\tau_{\leq 1}Rj_*\mathbb{Z}/p^n(1)_{X_K}^\prime)$$
have kernel and cokernel annihilated by some $p^N$, where $N$ only depends on $\pi:\mathcal{X}\longrightarrow S$ and is independent of $n$.
So do \\
$$
H^2_{\et}(Y,\tau_{\leq 1}\mathscr{S}_n(1)_{\mathcal{X}})\longrightarrow H^2_{\et}(Y,\mathscr{S}_n(1)_{\mathcal{X}}) $$
and 
$$H^2_{\et}(Y,\tau_{\leq 1}\mathscr{S}_n(1)_{\mathcal{X}})\longrightarrow H^2_{\et}(Y,i^*\tau_{\leq 1}Rj_*\mathbb{Z}/p^n(1)_{X_K}^\prime). $$
Taking limit, we get an isomorphism\\
$$\varprojlim_n H^2_{\et}( Y,\mathscr{S}_n(1)_{\mathcal{X}})\otimes_\ZZ \mathbb{Q}  \cong\varprojlim_n H^2_{\et}(Y,i^*\tau_{\leq 1}Rj_*\mathbb{Z}/p^n(1)_{X_K}^\prime)\otimes_\ZZ \mathbb{Q}.
$$
By the proper base change theorem, we have
$$H^2_{\et}(\mathcal{X},\tau_{\leq 1}Rj_*\mathbb{Z}/p^n(1)_{X_K}^\prime)\cong H^2_{\et}(Y,i^*\tau_{\leq 1}Rj_*\mathbb{Z}/p^n(1)_{X_K}^\prime).$$
There is a natural map\\
$$ H^2_{\syn}(\mathcal{X},\mathbb{Q}_p(1))  \longrightarrow \varprojlim_n H^2_{\et}( Y,\mathscr{S}_n(1)_{\mathcal{X}})\otimes_{\ZZ} \mathbb{Q}.
$$
By composition, we get a natural map
$$ H^2_{\syn}(\mathcal{X},\mathbb{Q}_p(1))\lra \varprojlim_n H^2_{\et}(\mathcal{X},\tau_{\leq 1}Rj_*\mathbb{Z}/p^n(1)_{X_K}^\prime)\otimes_{\ZZ} \mathbb{Q}.
$$
By Lemma $\ref{lem15}$, the natural map
$$H^2_{\syn}(\mathcal{X},\mathbb{Q}_p(1))\lra H^2_{\et}(X_{\bar{K}},\mathbb{Q}_p(1))^{G_K}
$$
is surjective.
By the following commutative diagram
\begin{displaymath}
\xymatrix{ 
H^2_{\syn}(\mathcal{X},\mathbb{Q}_p(1))\ar[dr]\ar[r] & \varprojlim_n H^2_{\et}(\mathcal{X},\tau_{\leq 1}Rj_*\mathbb{Z}/p^n(1)_{X_K}^\prime)\otimes_{\ZZ} \mathbb{Q} \ar[d]  \\
  & H^2_{\et}(X_{\bar{K}},\mathbb{Q}_p(1))^{G_K}}
\end{displaymath}
the map
$$\varprojlim_n H^2_{\et}(\mathcal{X},\tau_{\leq 1}Rj_*\mathbb{Z}/p^n(1))\otimes_{\ZZ}\mathbb{Q}\longrightarrow H^2_{\et}(X_{\bar{K}},\mathbb{Q}_p(1))^{G_K}$$
is surjective.
\end{proof}
\begin{lem}\label{thm11}
 Assuming that $R$ has a finite residue field, let $X$ be a smooth proper variety over $K$.  Assuming the conjecture of purity of the weight filtration \emph{(\cite[Conj. 2.6.5)]{Illu})}(which is known for $\dim(X)\leq 2$),  then the eigenvalues of the geometric Frobenius action on $H^{i}_{\et}(X_{\bar{K}},\mathbb{Q}_p(m))^I$ are Weil numbers of weight $\leq$ $i-2m$.
\end{lem}
\begin{proof}
By de Jong's alteration theorem \cite[Cor. 5.1]{deJ1}, there exists an alteration of $X$ with semi-stable reduction over $R$. Without loss of generality, we can assume that there exists a proper semi-stable morphism $\pi:\mathcal{X}\longrightarrow \Spec R$ whose generic fiber is identified with $X/K$.

Replace $K$ by its completion. Let $K_0$ ( resp. $K_1$) denote the fraction field of $W(k)$( resp. $W(\bar{k})$). By the semi-stable comparison theorem, there is a canonical isomorphism
$$H^i_{\et}(X_{\bar{K}},\mathbb{Q}_p)\otimes_{\mathbb{Q}_p}B_{st}\cong H^i_{0}(\mathcal{X})\otimes_{K_0}B_{st}$$
compatible with actions of $G_K$, Frobenius $\phi$, $N$ and the filtration after tensoring $K$ on both sides, where the $\phi$ action on right side is given by $\phi_{st}\otimes \phi_B$ and $N$ acts by $N\otimes 1+1\otimes N$ on the right side. Set $V=H^i(X_{\bar{K}},\mathbb{Q}_p)$. We also write $D_{st}(V)$ for $H^i_{0}(\mathcal{X})$. 
Note $D_{st}(V(m))=D_{st}(V)(m)$.  And we have\\
$$
V(m)=\{x\in D_{st}(V(m))\otimes_{K_0} B_{st}| \phi x=x, Nx=0, 1\otimes_{K_0} x \in Fil^0\}.
$$
Since $B_{st}^I=K_1$ and $B_{cris}=\Ker (N: B_{st}\longrightarrow B_{st})$, we have\\
$$V(m)^I=\{x\in D_{st}(V(m))\otimes_{K_0} K_1| \phi x=x, Nx=0 \} \cap Fil^0.
$$
Since $N\phi=p\phi N$ on $D_{st}(V(m))$,  so $\phi$ keeps $(D_{st}(V(m))^{N=0}$. $N$ acts as zero on $K_1$, therefore\\
$$
V(m)^I\subset\{x\in D_{st}(V(m))^{N=0}\otimes_{K_0} K_1| \phi x=x\}.
$$
 Assuming $k=\mathbb{F}_{p^n}$, there is a $\phi^n$-equivariant filtration $M_i$ on $D_{st}(V)$ determined by $N$. Since $N$ acts on $D_{st}(V(m))$ same as action on $D_{st}(V)$ (identifying the underlying space) and $\phi_{D_{st}(m)}=p^{-m}\phi_{D_{st}} $, we have
$$ 
\{x\in D_{st}(V(m))^{N=0}\otimes_{K_0} K_1| \phi x=x\}=\{ x\in D_{st}(V)^{N=0}\otimes_{K_0} K_1|  \phi^n(x)=p^{mn}x\}.
$$
$\phi^n$ can be written as $(\phi^n_{st}\otimes 1) \circ (1\otimes \phi^n_{B})$ as $K_0$-linear maps on $D_{st}(V)^{N=0}\otimes_{K_0} K_1$. Set $A=\phi^n_{st}\otimes 1$ and $B=1\otimes \phi^n_{B}$, then $\phi^n=AB=BA$. Set \\
$$U=\{ x\in D_{st}(V)^{N=0}\otimes_{K_0} K_1|  \phi^n(x)=p^{mn}x\}.$$
It is an invariant $\mathbb{Q}_p$-linear subspace of $A$ and $B$. Since $D_{st}(V)^{N=0}\subseteq M_0$, by \cite[Conj. 2.6.6(a)]{Illu}, all eigenvalues of $\phi^n$ on $D_{st}(V)^{N=0}$ are Weil numbers of weight $\leq i$. So there is polynomial $P(X)$ ( the product of Galois conjugates of the characteristic polynomial of $\phi^n$) in $\mathbb{Q}[X]$ such that $P(\phi^n)=0$ on $D_{st}(V)^{N=0}$. So $P(B^{-1}p^{mn})=0$ on $U$. Set $G(X)=X^{\deg(P)}P(p^{mn}X^{-1})$, we have $G\in \mathbb{Q}[X]$ and $G(B)=0$ on $U$.  Actually, the $B$ action on $U$ coincides with the action of the arithmetic Frobenius $F\in \Gal(K^{sh}/K)$ on $U$. Therefore, $H^i(X_{\bar{K}},\mathbb{Q}_p(m))^I$ is a $B$-invariant subspace of $U$. Therefore, $G(F)=0$ on $H^i(X_{\bar{K}},\mathbb{Q}_p(m))^I$. Since all roots of $P(X)$ are Weil numbers of weight $\leq i$, all roots of $G(X)$ are Weil numbers of weight $\geq 2m-i$.
\end{proof}

\begin{rem}
 The conjecture of purity of the weight filtration is known to hold for the case dim $X\leq 2$. For $i\leq 2$, the statement in the theorem can be reduced to the case dim $X\leq 2$ using a Lefschetz hyperplane argument, thus holding unconditionally for $i\leq2$. In the case that $X$ has a good reduction, the eigenvalues of the geometric Frobenius action on $H^{i}_{\et}(X_{\bar{K}},\mathbb{Q}_p(m))^I$ are Weil numbers of pure weight $i-2m$. A sketch of the above proof can be found in \cite[\S 5]{Jan}.
\end{rem}

\begin{lem}\label{lem16}
Let $\pi:\mathcal{X}\longrightarrow S=\Spec R$ be a semi-stable flat projective morphism. Assume that it comes from a base change of some $\mathcal{X}_0\longrightarrow S_0=\Spec R_0$, where $R_0$ is a Henselian DVR with a finite residue field $k_0$ and $R=R^{sh}_0$. Let $\mathfrak{T}_n(1)\in D^b(\mathcal{X}_{\et},\mathbb{Z}/p^n\mathbb{Z})$ be the $p$-adic etale Tate twist defined by K. Sato in \emph{\cite[\S1.3]{Sat}} ( which is naturally isomorphic to $\mathbb{G}_m\otimes^\LL\mathbb{Z}/p^n\mathbb{Z}[-1]$). Then the natural map\\
$$\varprojlim_n H^2_{\et}(\mathcal{X},\mathfrak{T}_n(1))\otimes_\ZZ \mathbb{Q}\longrightarrow H^2_{\et}(X_{\bar{K}},\mathbb{Q}_p(1))^{G_K}$$
is surjective.
\end{lem}
\begin{proof}
By the definition of $\mathfrak{T}_n(m)$, there is a distinguished triangle in $D^b(\mathcal{X}_{\et},\mathbb{Z}/p^n\mathbb{Z})$\\
\begin{equation}\label{eqsato}
i_*\nu_{Y,n}^{m-1}[-m-1]\longrightarrow \mathfrak{T}_n(m)\longrightarrow \tau_{\leq m}Rj_*\mathbb{Z}/p^n(m) \longrightarrow
i_*\nu_{Y,n}^{m-1}[-m],
\end{equation}
where $\nu_{Y,n}^m$ are generalized Logarithmic Hodge–Witt sheaves which agree with $W_n\Omega_{Y,\log}^{m}$ if $Y$ is smooth (cf. \cite[Lem. 1.3.1]{Sat}). $\nu_{Y,n}^0$ can be identified with 
$$
\bigoplus_{y\in Y^0}i_{y*}\mathbb{Z}/p^n\mathbb{Z}
$$
where $Y^0$ is the set of generic points of $Y$. Taking $m=1$ in (\ref{eqsato}) and taking cohomology, we get an exact sequence\\
$$\varprojlim_n H^2_{\et}(\mathcal{X},\mathfrak{T}_n(1))\otimes_\ZZ \mathbb{Q} \longrightarrow \varprojlim_n H^2_{\et}(\mathcal{X},\tau_{\leq 1}Rj_*\mathbb{Z}/p^n(1))\otimes_\ZZ \mathbb{Q}\longrightarrow 
\varprojlim_n H^1_{\et}(Y,\nu_{Y,n}^0)\otimes_\ZZ \mathbb{Q}.
$$
$G_{k_0}$ acts on the above sequence, and we will show the third term is of pure weight $1$. Let $\widetilde{Y}\longrightarrow Y $be an alteration such that $\widetilde{Y}$ is smooth. It induces a map\\
$$
H^1_{\et}(Y,\nu_{Y,n}^0)\longrightarrow H^1_{\et}(\widetilde{Y},\nu_{\widetilde{Y},n}^0).
$$
Since  $ H^1_{\et}(Y,\nu_{Y,n}^0)\hookrightarrow H^1_{\et}(Y^0,\mathbb{Z}/p^n\mathbb{Z})$ and  $H^1_{\et}(\widetilde{Y},\nu_{\widetilde{Y},n}^0)\hookrightarrow H^1_{\et}(\widetilde{Y}^0,\mathbb{Z}/p^n\mathbb{Z})$, and the kernel of $H^1_{\et}(Y^0,\mathbb{Z}/p^n\mathbb{Z}) \longrightarrow  H^1_{\et}(\widetilde{Y}^0,\mathbb{Z}/p^n\mathbb{Z})$ is killed by some positive integer independent of $n$, we have an injection\\
$$\varprojlim_n H^1_{\et}(Y,\nu_{Y,n}^0)\otimes_\ZZ \mathbb{Q}\hookrightarrow \varprojlim_n H^1_{\et}(\widetilde{Y},\nu_{\widetilde{Y},n}^0)\otimes_\ZZ \mathbb{Q}.$$
Therefore, the natural map
$$\varprojlim_n H^1_{\et}(Y,\nu_{Y,n}^0)\otimes_\ZZ \mathbb{Q} \lra \varprojlim_nH^1_{\et}(\widetilde{Y},\nu_{\widetilde{Y},n}^0)\otimes_\ZZ\mathbb{Q}$$
is injective.
We may assume that the irreducible components of $Y$ and $\widetilde{Y}$ are defined over $k_0$. Therefore, $H^1_{\et}(\widetilde{Y},\nu_{\widetilde{Y},n}^0)$ admits a $\Gal(k/k_0)$-action. Since $\widetilde{Y}$ is smooth over $k$, we have $\nu_{\widetilde{Y},n}^0=\mathbb{Z}/p^n\mathbb{Z}$. Thus
 $$\varprojlim_nH^1_{\et}(\widetilde{Y},\nu_{\widetilde{Y},n}^0)\otimes_\ZZ\mathbb{Q}=H^1_{\et}(\widetilde{Y},\mathbb{Q}_p).
 $$
  Since $H^1_{\et}(\widetilde{Y},\mathbb{Q}_p)$ is of pure weight $1$ (cf. \cite[\S 3]{Jan}), $\varprojlim_n H^1_{\et}(Y,\nu_{Y,n}^0)\otimes_\ZZ \mathbb{Q}$ is also of pure weight $1$.

By Proposition $\ref{thm13}$, 
$$
\varprojlim_n H^2_{\et}(\mathcal{X},\tau_{\leq 1}Rj_*\mathbb{Z}/p^n(1))\otimes_\ZZ\mathbb{Q}\longrightarrow H^2_{\et}(X_{\bar{K}},\mathbb{Q}_p(1))^{G_K}
$$
is surjective. By Lemma $\ref{thm11}$, $H^2_{\et}(X_{\bar{K}},\mathbb{Q}_p(1))^{G_K}$ is of mixed weight $\leq 0$. Therefore\\
$$\varprojlim_n H^2_{\et}(\mathcal{X},\mathfrak{T}_n(1))\otimes_\ZZ \mathbb{Q}\longrightarrow H^2_{\et}(X_{\bar{K}},\mathbb{Q}_p(1))^{G_K}$$
is surjective.
\end{proof}

Proof of Theorem \ref{main}
\begin{proof}
Let $f:\CX_{\fppf}\lra \CX_{\et}$ be the morphism between Grothendieck topologies. By \cite[Thm. 3.9]{Mil1}, $R^if_*\GG_m=0$ for $i>0$. By the Kummer exact sequence of $\fppf$-sheaves
 $$0\lra\mu_{p^n}\lra\GG_m\stackrel{p^n}{\lra}\GG_m\lra 0,$$
the complex $\GG_m\stackrel{p^n}{\lra}\GG_m$ can be viewed as an acyclic resolution of $\mu_{p^n}$. Hence,
$$Rf_*\mu_{p^n}\cong \mathbb{G}_m\otimes^\LL\mathbb{Z}/p^n\mathbb{Z}[-1].$$
Since $\mathfrak{T}_n(1)\cong \mathbb{G}_m\otimes^\LL\mathbb{Z}/p^n\mathbb{Z}[-1]$ (cf. \cite[Prop. 4.5.1]{Sat}), we have 
$$ H^2_{\fppf}(\CX_{R^{sh}}, \mu_{p^n})=H^2_{\et}(\CX_{R^{sh}},Rf_*\mu_{p^n})\cong  H^2_{\et}(\CX_{R^{sh}},\mathfrak{T}_n(1)).
 $$
Then the claim follows from Lemma \ref{lem16}. 
\end{proof}
\subsection{Proof of Proposition \ref{cor1}}

\begin{lem}\label{lem18}
Let $\pi:\mathcal{X}\longrightarrow S=\Spec R$ be a proper flat morphism, where $R$ is an excellent strictly Henselian DVR with an algebraically closed residue field. Let $K$ denote the quotient field of $R$. Assuming that $\mathcal{X}$ is regular and the generic fiber $X$ is smooth projective geometrically connected over $K$, then the map
$$ 
\Br(\mathcal{X})\longrightarrow \Br(X_{K^s})^{G_K}
$$
has a kernel of finite exponent. The claim still
holds for the prime-to-$p$ part if the residue field of $R$ is only separably closed.
\end{lem}
\begin{proof}
If $\dim(\mathcal{X})\leq 2$, Artin's theorem \cite[Thm.3.1]{Gro3} implies that both groups vanish. If $\dim(\CX)>2$, by Lemma \ref{thekeytrick}, there are smooth projective integral curves $Z_i \subset X$, a continuous $G_K$-module $A$, and a $G_K$-equivariant map
$$\Pic(X_{K^s})\times A\lra \bigoplus_i\Pic(Z_{i,K^s})$$
with a kernel and a cokernel of finite exponent. It follows that the natural map
$$H^1(K,\Pic_{X/K})\lra \bigoplus_i H^1(K,\Pic_{Z_i/K})$$
has a kernel of finite exponent. By taking the Zariski closures of $Z_i$ in $\CX$ and then desingularizing them, we can assume that there are proper flat morphisms $\pi_i:\mathcal{Z}_i\longrightarrow S$ such that $\CZ_i$ is regular and the generic fiber $Z_i$ is smooth over $K$ and $S$-morphisms $\mathcal{Z}_i\longrightarrow \mathcal{X}$.
Consider the  commutative diagram \\
\begin{displaymath}
\xymatrix{ 
0 \ar[r] & H^1(K,\Pic_{X/K}) \ar[r]\ar[d] & \Br(X)\ar[r]\ar[d] &   \Br(X_{K^s})^{G_K} \ar[d]
\\
0 \ar[r] & \bigoplus_i H^1(K,\Pic_{Z_i/K}) \ar[r] & \bigoplus_i \Br(Z_i)\ar[r] &  \bigoplus_i \Br(Z_{i,K^s})^{G_K}  }
\end{displaymath}
The exact rows arise from the spectral sequences\\
$$
H^p(K,H^q(X_{K^s},\mathbb{G}_m))\Rightarrow H^{p+q}(X,\mathbb{G}_m),
$$
and the fact that $H^q(K, H^0(X_{K^s},\mathbb{G}_m))=0, q>0$. Since $\Br(\CZ_i)=0$, $\Br(\CX)\cap H^1(K,\Pic_{X/K})$ will be mapped to zero in $\Br(Z_i)$. Thus, $\Br(\CX)\cap H^1(K,\Pic_{X/K})$ is contained in the kernel of the first column, which has finite exponent.
\end{proof}

Proof of Proposition \ref{cor1}.
\begin{proof}
It follows directly from the preceding lemma that the natural map
$$\Br(\CX_{R^{sh}})(p)\lra\Br(X_{\Bar{K}})^{I}(p)$$
has a kernel of finite exponent. By Proposition \ref{prop2.1}, the natural map
$$H^2_{\et}(X_{\bar{K}},\QQ_p(1))^{I} \lra V_p\Br(X_{\bar{K}})^{I}
$$
is surjective. By Theorem \ref{main}, the natural map
$$V_p\Br(\CX_{R^{sh}})\lra V_p\Br(X_{\Bar{K}})^{I}$$
is surjective. Since $\Br(X_{\Bar{K}})^{I}$ is of cofinite type,
$$\Br(\CX_{R^{sh}})(p)\lra\Br(X_{\Bar{K}})^{I}(p)$$
has a finite cokernel.
\end{proof}
\subsection{Proof of Corollary \ref{big}}
We will deduce the finiteness of the kernel and the $p$-primary part of the cokernel from Proposition \ref{cor1}. The finiteness of the prime-to-$p$-parts of the kernel and the cokernel was proved by Colliot-Thélène and Saito in \cite[Cor. 2.6]{CTSa} (cf. Remark \ref{rem11}). To prove the finiteness of the cokernel, we will use a Bertini type theorem for strictly semi-stable morphisms over DVRs, as developed in \cite{JS}, along with a pullback trick introduced by Colliot-Thélène and Skorobogatov (cf. \cite{CTS1} or \cite{Yua}), which allows us to reduce the problem to the case of relative dimension $1$.
\begin{lem}\label{lem2.2} 
Let $\pi:\mathcal{X}\longrightarrow S=\Spec R$ be a proper flat map, where $R$ is an excellent Henselian DVR with a fraction field $K$ and a finite residue field $k$ of characteristic $p>0$. Assuming that $\mathcal{X}$ is regular and the generic fiber $X$ of $\pi$ is geometrically connected over $K$, then there is an exact sequence\\
$$\Br(\mathcal{X})\longrightarrow \Br(\CX_{R^{sh}})^{G_k}\longrightarrow H^2(S,R^1\pi_*\mathbb{G}_m),$$
where the first map in the sequence has a finite kernel.
\end{lem}
\begin{proof}
By the Leray spectral sequence 
$$ 
E_2^{p,q}=H^p(S,R^q\pi_*\mathbb{G}_m)\Rightarrow H^{p+q}(\mathcal{X},\mathbb{G}_m),
$$
we get a long exact sequence
$$H^2(S,\mathbb{G}_m)\longrightarrow \Ker(H^2(\mathcal{X},\mathbb{G}_m)
\longrightarrow H^0(S,R^2\pi_*\mathbb{G}_m))\longrightarrow H^1(S,R^1\pi_*\mathbb{G}_m)\longrightarrow H^3(S,\mathbb{G}_m)
$$
By \cite[Chap. II, Prop. 1.5]{Mil1}, $H^i(S,\mathbb{G}_m)=0$ for all $i\geq 1$. Since $H^0(C,R^2\pi_*\mathbb{G}_m)=\Br(\CX_{R^{sh}})^{G_k}$, we get the desired exact sequence. To show that the kernel of the first map is finite, it suffices to show that $H^1(S,R^1\pi_*\mathbb{G}_m)$ is finite. We have
$$H^1(S,R^1\pi_*\mathbb{G}_m)=H^1(G_k, \Pic(\CX_{R^{sh}})).$$
Since the natural map $\Pic(\CX_{R^{sh}}) \rightarrow \Pic(X_{K^{sh}})$ is surjective and has a finitely generated cokernel, it suffices to show that
$H^1(G_k, \Pic(X_{K^{sh}}))$ is finite. Since $\Br(K^{sh})=0$, we have 
$$\Pic(X_{K^{sh}})=\Pic_{X/K}(K^{sh}).$$
Since the cokernel of $\Pic^0_{X/K}(K^{sh})\rightarrow \Pic_{X/K}(K^{sh})$ is finitely generated, it suffices to show that $H^1(G_k, \Pic(X_{K^{sh}}))$ is finite. This follows from \cite[Chap. I, Prop.3.8]{Mil1}.
\end{proof}
By Lemma \ref{lem2.2}, the kernel of the natural map $\Br(\CX_{R^{sh}})\rightarrow \Br(X_{\Bar{K}})$ has finite exponent. By the above lemma, the kernel of the natural map $\Br(\CX)\rightarrow \Br(X_{\bar{K}})^{G_K}$ also has finite exponent and is actually finite, since $\Br(X)$ is of cofinite type. To show that
the natural map $\Br(\CX)\rightarrow \Br(X_{\bar{K}})^{G_K}$ has a finite cokernel, it suffices, by Proposition \ref{cor1}, to show that the natural map
$$\Br(\CX_{R^{sh}})^{G_k}\longrightarrow H^2(S,R^1\pi_*\mathbb{G}_m)$$
has an image of finite exponent. In the case where $\dim(\CX)=2$, this is clear since $\Br(\CX_{R^{sh}})=0$. More generally, we will use the pullback trick again to reduce the problem to the case of relative dimension $1$.
\begin{lem}\label{Bertini}
Let $\pi:\mathcal{X}\longrightarrow S=\Spec R $ be a projective flat morphism, where $R$ is a strictly Henselian DVR with a residue field $k$ and a quotient field $K$. Assume that $\mathcal{X}$ is strictly semi-stable over $S$ ( cf. \cite[Def. 1.1]{JS} ) and that its generic fiber is geometrically connected over $K$. Let $Y$ be the special fiber of $\pi$. Write $Y=\cup_{i=1}^{n}Y_i$ as the union of irreducible components (by the assumption, $Y_i$ are smooth projective connected varieties over $k$ ). Assume $m=\dim(Y)\geq 1$. Fix an $S$-embedding $\mathcal{X}\hookrightarrow \mathbb{P}^n_{S}$. Then there exist hyperplanes $H_1,..., H_{m-1}$ defined over $R$ such that the scheme-theoretic intersections $\mathcal{Z}= (\cap_{i=1}^{m-1} H_i)\cap \CX$ satisfies the same assumption as $\mathcal{X}\longrightarrow S$ ( strictly semi-stable and has a geometrically connected generic fiber) and the scheme-theoretic intersection $C_j=Y_j\cap(\cap_{i=1}^{m-1} H_i)$ are distinct smooth projective connected curves.
\end{lem}
\begin{proof}
If $\dim(Y)=1$, the claim is trivial. Now, assume $m=\dim(Y)>1$, by Bertini's theorem for strictly stable-schemes over a DVR \cite[Thm. 1.2 and Lem. 1.3]{JS}, there exists a hyperplane $H$ over $S$ such that $H$ intersects $Y_j$ and $Y_i\cap Y_j$ transversely in $\mathbb{P}^n_k$, and $H\cap\mathcal{X}$ is strictly semi-stable over $S$. Since $\dim(Y)>1$ and $k$ is separably closed, by \cite[Chap. III, Cor. 7.9]{Har}, the scheme-theoretic intersections $H\cap Y_j$ and $H_{\bar{K}}\cap X_{\bar{K}}$ are smooth and connected.  Since $H$ intersects $Y_i\cap Y_j$  transversely, $H\cap Y_i$ are distinct irreducible components of the special fiber of $\mathcal{X}\cap H$.  By induction on the dimension , the claim is true for $\mathcal{X}\cap H$, so there exit $H_2,..., H_{m-1}$ satisfying conditions in the claim for $\mathcal{X}\cap H$. Then $H_1=H, H_2, ..., H_{m-1}$ are desired hyperplanes for $\mathcal{X}$.  
\end{proof}

\begin{lem}\label{intersction}
Let $\pi:\CZ\lra S=\Spec R$ be a proper flat morphism where $R$ is a strictly Henselian DVR and $\CZ$ is an integral regular scheme of dimension $2$. Let $C_i$ denote the irreducible components of the special fiber $C$. Then the intersection pairing 
$$\bigoplus_i\QQ\cdot [C_i] \times \bigoplus_i\QQ\cdot [C_i] \lra \QQ$$
has both the left and the right kernels spanned by $[C]$.
\end{lem}
\begin{proof}
For a detailed proof, refer to \cite[Chap. 9, Thm. 1.23]{Liu}.
\end{proof}
\begin{lem}\label{alsplit}
Let $G$ be a group and $M,\  L$ be $G$-modules. We say a $G$-map $f:M\lra L$ is almost split if there exists a $G$-map $N\lra L$ such that $M\oplus N\lra L$ has kernel and cokernel of finite exponent. Assume that there is a commutative diagram with exact rows
\begin{displaymath}
\xymatrix{
	0\ar[r] &A_1\ar[r]\ar[d] & B_1\ar[r]\ar[d] &C_1\ar[r]\ar[d] & 0\\
0\ar[r] & A_2\ar[r]& B_2\ar[r] &C_2\ar[r] & 0}
\end{displaymath}
such that the first column, the third column and $A_2\rightarrow B_2$ are almost split. Then the middle column is almost split.
\end{lem}
\begin{proof}
Let $A^\prime_1\lra A_2$, $A^\prime_2 \lra B_2$ and  $C^\prime_1\lra C_2$ be $G$-maps such that 
$A_1^\prime \oplus A_1 \lra A_2$, $A_2\oplus A^\prime_2 \lra B_2$ and $C_1\oplus C^\prime_1\lra C_2$ have kernels and cokernels of finite exponent. We may assume $A^\prime_2 \subseteq B_2$ and  $C^\prime_1\subseteq C_2$. Then $A^\prime_2\lra C_2$ has a kernel and a cokernel of finite exponent. Let $B^\prime $ denote the preimage of $C^\prime_1$ under this map. One can show that $B^\prime \lra C^\prime_1$ has a kernel and a cokernel of finite exponent. Replace $C^\prime_1$ by the image of this map, we may assume that the map $B^\prime \lra C^\prime_1$ is surjective. Let $K$ denote its kernel, so $K\subseteq A_2$ and is of finite exponent. Consider the diagram
\begin{displaymath}
\xymatrix{
	0\ar[r] &A_1^\prime\oplus A_1\oplus K\ar[r]\ar[d] &A_1^\prime\oplus B_1\oplus B^\prime \ar[r]\ar[d] &C_1\oplus C^\prime_1 \ar[r]\ar[d] & 0\\
0\ar[r] & A_2\ar[r]& B_2\ar[r] &C_2\ar[r] & 0}
\end{displaymath}
where the first and the third column has kernels and cokernels of finite exponent. By the snake lemma, the second column has a kernel and a cokernel of finite exponent.
\end{proof}

\begin{prop}\label{plocalsurj}
Let $\pi:\mathcal{X}\longrightarrow S=\Spec R$ be a projective flat morphism with a smooth and geometrically connected generic fiber $X$, where $R$ is an excellent Henselian DVR. Assuming that the base change of $\pi$ to $R^{sh}$ is strictly semi-stable, then
$$\Br(\mathcal{X})\longrightarrow \Br(\mathcal{X}_{R^{sh}})^{G_{k}}$$
has a kernel and a cokernel of finite exponent.
\end{prop}
\begin{proof}
By Lemma \ref{lem2.2}, it suffices to show that the cokernel is of finite exponent. If $\pi$ is of relative dimension $1$,  by Artin's theorem \cite[Thm.3.1]{Gro3}, $\Br(\CX_{R^{sh}})=0$, so the claim is true in this case. For the general case, we will use the pullback trick to reduce the problem to the case of relative dimension $1$. Consider the commutative diagram 
\begin{displaymath}
\xymatrix{
\CZ\ar[r]\ar[d]^{\pi^\prime} & \CX \ar[d]^{\pi}\\
 S \ar[r]^{id}  &  S}
\end{displaymath}
where $\pi^\prime$ is proper flat and of relative dimenson $1$. By Proposition \ref{lem2.2}, there is commutative diagram
\begin{displaymath}
\xymatrix{
\Br(\mathcal{X})\ar[r]\ar[d] &\Br(\CX_{R^{sh}})^{G_k}\ar[r]^{\partial}\ar[d] & H^2(S,R^1\pi_*\mathbb{G}_m) \ar[d]^{a}\\
\Br(\mathcal{Z})\ar[r] &\Br(\CZ_{R^{sh}})^{G_k}\ar[r] & H^2(S,R^1\pi^\prime_*\mathbb{G}_m).}
\end{displaymath}
Since $\Br(\CZ_{R^{sh}})^{G_k}=0$, we have $\Im(\partial)\subseteq \Ker(a)$. The idea is to show that $\Ker(a)$ is of finite exponent. Instead of working on one $\CZ$,  we will find finitely many $\CZ_i$ and show that the kernel of
$$H^2(S,R^1\pi_*\mathbb{G}_m)\stackrel{\sum a_i}\lra \bigoplus_{i} H^2(S,R^1\pi^\prime_*\mathbb{G}_m)
$$
is of finite exponent. By the Hochschild–Serre spectral sequence
$$ H^p(G_{k},H^q(\Spec R^{sh}, R^1\pi_*\mathbb{G}_m))\Rightarrow H^{p+q}(S, R^1\pi_*\mathbb{G}_m)),
$$
we get \\
$$H^2(S,R^1\pi_*\GG_m)=H^2(G_k,H^0(\Spec R^{sh}, R^1\pi_*\mathbb{G}_m))=H^2(G_k,\Pic(\CX_{R^{sh}})).
$$
Thus, the map $a$ can be identified with
$$H^2(G_k,\Pic(\CX_{R^{sh}})) \lra H^2(G_k,\Pic(\CZ_{R^{sh}})).$$
Let $Y$ denote the special fiber of $\pi$ and $Y_i$ denote its irreducible components. Then there is an exact sequence
$$ \bigoplus_i \ZZ\cdot[Y_i]\lra \Pic(\CX_{R^{sh}}) \lra \Pic(X_{K^{sh}})\lra 0.$$
The kernel of the first arrow is generated by $[Y]$, this actually follows from the following long exact sequence (cf.\cite[\S 6]{Gro3})
$$ H^0(\CX_{R^{sh}},\GG_m) \lra H^0(X_{K^{sh}},\GG_m)\lra H^1_{Y}(\CX_{R^{sh}},\GG_m)\lra H^1(\CX_{R^{sh}},\GG_m)\lra H^1(X_{K^{sh}},\GG_m),
$$
and the facts 
$$H^0(\CX_{R^{sh}},\GG_m)=(R^{sh})^*, \  H^0(X_{K^{sh}},\GG_m)=(K^{sh})^{*},$$
and 
$$H^1_{Y}(\CX_{R^{sh}},\GG_m)=\bigoplus_i \ZZ\cdot[Y_i].$$
Thus, we have 
$$ 0\lra (\bigoplus_i \ZZ\cdot[Y_i])/\ZZ\cdot [Y]\lra \Pic(\CX_{R^{sh}}) \lra \Pic(X_{K^{sh}})\lra 0.$$
We will denote $(\bigoplus_i \ZZ\cdot[Y_i])/\ZZ\cdot [Y]$ by $D_\CX$. Let $\CZ\lra S$ be a proper flat morphism with $\CZ$ an integral regular scheme of dimension $2$. Let $D_\CZ$ denote the subgroup of $\Pic(\CZ_{R^{sh}})$ generated by prime divisors contained in the special fiber. 
Consider the $G_k$-equivariant intersection pairing
\begin{equation}\label{interequ}
\Pic(\CZ_{R^{sh}})\times D_{\CZ} \lra \ZZ,
\end{equation}
by Lemma \ref{intersction}, the restriction of the pairing to 
$D_{\CZ}\times D_{\CZ}$ is perfect after tensoring $\QQ$ (Here $\CZ_{R^{sh}}$ may be disconnected, the claim is still true since it is true for each connected component). Let $D_{\CZ}^{\perp}$
denote the left kernel of the pairing (\ref{interequ}).  Then the map
$$D_{\CZ}\oplus D_{\CZ}^{\perp}\lra \Pic(\CZ_{R^{sh}})$$
has a kernel and a cokernel of finite exponent. The kernel can be identified with the left kernel of the paring  $D_{\CZ}\times D_{\CZ} \lra \ZZ$ which is finite. Since $D_{\CZ}\lra \Hom(D_{\CZ},\ZZ)$ has a cokernel killed by some positive integer $m$, for any element $v\in \Pic(\CZ_{R^{sh}})$, $v$ defines an element $v^*$ in $\Hom(D_{\CZ},\ZZ)$, so there exists an element $u\in D_{\CZ}$ such that $mv^*=u^*$. Thus $mv-u \in D_{\CZ}^{\perp}$. This proves that the cokernel is killed by $m$. Thus, $D_{\CZ}\lra \Pic(\CZ_{R^{sh}})$ is almost split.

By Lemma \ref{Bertini}, we can choose $\CZ_0$ by taking hyperplane sections repeatedly such that $D_{\CX}\cong D_{\CZ_0}$. Then by Lemma \ref{thekeytrick} (ii), we can choose smooth projective integral curves $Z_i\subseteq X$ for $i=1,...,n$ such that 
$$\Pic(X_{K^s})\lra \bigoplus_{i=0}^{n} \Pic(Z_{i,K^s})$$
is almost split as $G_K$-modules. Taking $G_{K^{sh}}$-invariant, we have 
$$\Pic(X_{K^{sh}})\lra \bigoplus_{i=0}^{n} \Pic(Z_{i,K^{sh}})$$
is almost split as $G_k$-modules (note that $\Pic(X) \rightarrow\Pic_{X/k}(k)$ has a kernel and cokernel of finite exponent for any smooth projective variety $X$ over a field $k$).

Taking the Zariski closure of $Z_i$ in $\CX$, then desingularizes it, we get a $S$-morphsim $\CZ_i\lra \CX$ such that $\CZ_i\lra S$ is proper flat with generic fiber $Z_i$. Thus, we get the following commutative diagram

\begin{displaymath}
\xymatrix{
	0\ar[r]& D_{\CX}\ar[r]\ar[d]& \Pic(\CX_{R^{sh}}) \ar[r]\ar[d] &\Pic(X_{K^{sh}})\ar[r]\ar[d]& 0 \\
0\ar[r]& \bigoplus_{i=0}^{n}D_{\CZ_i}\ar[r] & \bigoplus_{i=0}^{n}\Pic(\CZ_{i,R^{sh}}) \ar[r]&\bigoplus_{i=0}^{n}\Pic(Z_{i,K^{sh}})\ar[r]& 0}
\end{displaymath}
By constructions, the first column (injective morphism between finitely generated abeliam groups with a continuous $G_k$-action), the third column and the second row are almost split, by Lemma \ref{alsplit}, the second column is also almost split. Therefore, 
$$H^2(G_k,\Pic(\CX_{R^{sh}}) )\lra \bigoplus_{i=0}^{n}H^2(G_k,\Pic(\CZ_{i,R^{sh}}))
$$
has a kernel of finite exponent. This completes the proof.
\end{proof}
\begin{prop}\label{cor22}
Let $\pi:\CX \lra S=\Spec R$ as above. Assuming that $\pi$ is strictly semi-stable and the residue field of $R$ is finite, then the natural map
$$
\Br(\mathcal{X})\longrightarrow \Br(X_{\bar{K}})^{G_{K}}
$$
has a finite kernel and a finite cokernel.
\end{prop}
\begin{proof}
By Proposition \ref{cor1},
$$\Br(\mathcal{X}_{R^{sh}})(p)\longrightarrow \Br(X_{\bar{K}})^{G_{K^{sh}}}(p)$$
has a kernel and a cokernel of finite exponent. It follows that the map 
$$\Br(\mathcal{X}_{R^{sh}})^{G_k}(p)\longrightarrow \Br(X_{\bar{K}})^{G_{K}}(p)$$
also has a kernel and a cokernel of finite exponent. By the lemma below, the claim also holds for the prime-to-$p$ part. Thus, the natural map 
$$\Br(\mathcal{X}_{R^{sh}})^{G_k}\longrightarrow \Br(X_{\bar{K}})^{G_{K}}$$
has a kernel and a cokernel of finite exponent. By Proposition \ref{plocalsurj},
$$
\Br(\mathcal{X})\longrightarrow \Br(X_{\bar{K}})^{G_{K}}
$$ 
has a kernel and a cokernel of finite exponent. Since both groups are of cofinite type, the kernel and cokernel are actually finite.
\end{proof}
\begin{rem}
To finish the proof of Corollary \ref{big}, we will remove the strictly semi-stable assumption in section \ref{remove}.
\end{rem}
\begin{lem}\label{lpart}
Let $\pi:\CX\lra S=\Spec R$ be a proper flat morphism where $R$ is a Henselian DVR with a finite residue field $k$ and a quotient field $K$. Assume that $\CX$ is regular and the generic fiber $X$ of $\pi$ is smooth projective geometrically connected over $K$. Let $\ell\neq \Char(k)$ be a prime. Then
$$\Br(\CX_{R^{sh}})^{G_k}(\ell)\lra\Br(X_{K^s})^{G_K}(\ell)$$
has a kernel and a cokernel of finite exponent and is an isomorphism for all but finitely many $\ell$. Moreover, if $\pi$ is smooth, the map is an isomorphism for all $\ell\neq \Char(k)$.
\end{lem}
\begin{proof}
Let $Y$ be the special fiber of $\pi$ and Let $Y_i$ be its irreducible components. We may assume that $Y_i$ is geometrically irreducible. By purity of Brauer group, there is an exact sequence

$$0\lra\Br(\CX_{R^{sh}})(\ell)\lra \Br(X_{K^{sh}})(\ell)\lra \bigoplus_{i} H^1(D_{i,k^s},\QQ_{\ell}/\ZZ_{\ell}),
$$
where $D_i$ is the smooth locus of $Y_i$. Thus, we get an exact sequence
$$0\lra\Br(\CX_{R^{sh}})^{G_k}(\ell)\lra \Br(X_{K^{sh}})^{G_k}(\ell)\lra \bigoplus_{i} H^1(D_{i,k^s},\QQ_{\ell}/\ZZ_{\ell})^{G_k}.
$$
By \cite[Lem. 2.4]{Qin1}, $H^1(D_{i,k^s},\QQ_{\ell}/\ZZ_{\ell})^{G_k}$ is finite  and vanishes for all but finitely many $\ell$.\\
Consider the spectral sequence
$$H^p(K^{sh}, H^q(X_{K^s},\GG_m))\Rightarrow H^{p+q}(X_{K^{sh}},\GG_m).$$
Since $H^p(K^{sh},\GG_m)=0$ for $p\geq 1$, we have 
$$0\lra H^1(K^{sh},\Pic_{X/K})\lra \Br(X_{K^{sh}})\lra \Br(X_{K^s})^{G_{K^{sh}}}.
$$
Taking $G_k$-invariant, we get 
$$0\lra H^1(K^{sh},\Pic_{X/K})^{G_k}\lra \Br(X_{K^{sh}})^{G_k}\lra \Br(X_{K^s})^{G_K}.
$$
By \cite[Thm. 2.1]{CTS1}, the last map has a cokernel of finite exponent. By \cite[Lem. 5.4]{Qin2}, the group
$H^1(K^{sh},\Pic_{X/K})^{G_k}(\ell)$ is finite and vanishes for all but finitely many $\ell$. This proves the first claim. The second claim follows from  \cite[Lem. 3.1(b)]{Qin1}.
\end{proof}

\subsection{Removing the semi-stable assumption}\label{remove}
We need the following two technical lemmas \ref{lem24} and \ref{lemgab} to remove the semi-stable assumption. 
\begin{lem}\label{lem24}
Let $f:Y\longrightarrow X$ be a proper morphism between regular noetherian schemes. Assume that $X$ is irreducible and each irreducible component of $Y$ dominates $X$, and that there is an open dense subset $U$ of $X$ such that $f|_{f^{-1}(U)}$ is a finite etale Galois covering over $U$. Assume $\Char(K(X))=0$. Set $G=\Aut(f^{-1}(U)/U)$. Then $G$ acts on $\Br(Y)$ and
the natural map \\
$$
\Br(X)\longrightarrow \Br(Y)^{G}
$$
has a kernel and a cokernel of finite exponent.
\end{lem}
\begin{proof}
Let $g \in G$, $g$ induces a $X$-rational map $Y\dashrightarrow Y$. Since $Y\longrightarrow X$ is proper, by the valuation criterion of properness, $g$ can extend to an open dense subset $V\subset Y$ with codim$(Y-V)\geq 2$. Thus, $g$ induces a map $\Br(Y)\longrightarrow \Br(V)$. By the purity of the Brauer group\cite{Ces}, $\Br(Y)=\Br(V)$, so $G$ acts on $\Br(Y)$. Let $K(Y)$ denote the product of the function fields of connected components of $Y$. By \cite[Thm. 1.2]{Ces}, there is an exact sequence\\
$$0 \longrightarrow \Br(Y) \longrightarrow \Br(K(Y))\longrightarrow \prod_i\prod_{y\in Y_i^1} \Br(K(Y_i))/\Br(\mathcal{O}_{Y,y}),
$$
where $Y_i^1$ denotes the set of point of codimension $1$ in the connected component $Y_i$ of $Y$.
Therefore we have the following commutative diagram\\

\begin{displaymath}
\xymatrix{ 
0 \ar[r] & \Br(X)\ar[r]\ar[d] & \Br(K(X))\ar[r]\ar[d] &   \prod \limits_{x\in X^1} \Br(K(X))/\Br(\mathcal{O}_{X,x}) \ar[d]
\\
0 \ar[r] & \mathcal{K} \ar[r] & \Br(K(Y))^G \ar[r]&  \prod \limits_{y\in f^{-1}(X^1)} \Br(K(\CO_{Y,y}))/\Br(\mathcal{O}_{Y,y})}
\end{displaymath}
Note that all points in $f^{-1}(X^1)$ has codimension $1$. Since $\Br(Y)^G\subset \mathcal{K}$, it suffices to show that the first column has a kernel and a cokernel of finite exponent. By the snake lemma, it is enough to show that the kernel of the third column is killed by $[K(Y):K(X)]^4$. Without loss of generality, we may assume that $X$ is the spectrum of a DVR.
Now, $f$ is a finite flat morphism ( cf. \cite[Prop. 14.123 and Cor. 12.89]{AG1v2} ).  
$Y$ is also the normalization of $X$ in $K(Y)$. Hence,  $\Gal(K(Y)/K(X))$ acts on $Y/X$. Consider the following commutative diagram
with exact rows.
\begin{displaymath}
\xymatrix{ 
0 \ar[r] & \Br(X)\ar[r]\ar[d] & \Br(K(X))\ar[r]\ar[d] & \Br(K(X))/\Br(X) \ar[d]\ar[r] & 0
\\
0 \ar[r] & \Br(Y)^G \ar[r] & \Br(K(Y))^G \ar[r]&  \Br(K(Y))/\Br(Y)}
\end{displaymath}
By the Hochschild-Serre spectral sequence associated with the action of $G$ on $Y$ ( cf. \cite[p. 18]{Qin1} ), the kernels and cokernels of the first two columns are killed by $[K(Y):K(X)]^2$. By the snake lemma, the kernel of the third column is killed by $[K(Y):K(X)]^4$.

\end{proof}
\begin{lem}\label{lemgab}\emph{(Gabber)}
Let $\pi:\mathcal{X}\longrightarrow S=\Spec R$ be a proper flat morphism, where $R$ is a Henselian DVR of characteristic $0$ with a perfect residue field. Let $K=\Frac(R)$ and $X$ denote the generic fiber of $\pi$. Assuming that $\mathcal{X}$ is regular and $X$ is geometrically connected over $K$, then there exists a strictly semi-stable projective morphism $\pi_1:\mathcal{X}_1\longrightarrow  S_1=\Spec R_1$, where $R_1$ is the ring of integers in some finite Galois extension $K_1$ of $K$, and an alteration $f:\mathcal{X}_1\longrightarrow \mathcal{X}\times_SS_1$ over $S_1$ such that $\mathcal{X}_1$ is irreducible and regular and $K(\mathcal{X}_1)/K(\mathcal{X})$ is a finite Galois extension.
\end{lem}
\begin{proof}
It follows from  \cite[Prop. 4.4.1]{Vidal}.
\end{proof}
\begin{thm}\label{thm25}
Let $\pi:\mathcal{X}\longrightarrow S=\Spec R$ be a proper flat morphism, where $R$ is a Henselian DVR of characteristic $0$ with a finite residue field.  Assuming that $\mathcal{X}$ is regular and the generic fiber $X$ of $\pi$ is geometrically connected over $K$, then the natural map
$$\Br(\mathcal{X})\longrightarrow \Br(X_{\bar{K}})^{G_{K}}$$ 
has a kernel and a cokernel of finite exponent.
\end{thm}
\begin{proof}
By the previous lemma, there is an alternation $f:\mathcal{X}_1\longrightarrow \mathcal{X}$ satisfying the assumptions in the Lemma \ref{lem24}. Let $X_1$ denote the generic fiber of $\pi_1$. Obviously, $X_1\otimes_{K}\bar{K}\longrightarrow X_{\bar{K}}$ also satisfies the assumptions in the Lemma \ref{lem24}. Set $G=\Gal(K(\mathcal{X}_1)/K(\mathcal{X}))$. By Lemma \ref{lem24}, the maps $\Br(\mathcal{X})\longrightarrow \Br(\mathcal{X}_1)^G$ and $\Br(X_{\bar{K}})\longrightarrow \Br(X_1\otimes _K\bar{K})^G$ have a kernel and a cokernel of finite exponent. Thus, the map
$$\Br(X_{\bar{K}})^{G_K}\longrightarrow (\Br(X_1\otimes _K\bar{K})^{G_K})^G$$
also has a kernel and a cokernel of finite exponent. Since 
$$X_1\otimes _K\bar{K}=\bigsqcup\limits_{\sigma:K_1\hookrightarrow \bar{K}} (X_1)_{\sigma},$$
so we have 
$$\Br(X_1\otimes _K\bar{K})^{G_K}=\Br((X_1)_{\bar{K}})^{G_{K_1}}.$$
Since $\mathcal{X}_1\longrightarrow S_1$ is strictly semi-stable, by Proposition \ref{cor22}, the natural map
$$\Br(\mathcal{X}_1)\longrightarrow \Br((X_1)_{\bar{K}})^{G_{K_1}}$$
has a kernel and a cokernel of finite exponent. It follows that 
$$\Br(\mathcal{X}_1)^G \lra (\Br(X_1\otimes _K\bar{K})^{G_K})^G$$
also has a kernel and a cokernel of finite exponent. Then the claim follows from the following commutative diagram\\
\begin{displaymath}
\xymatrix{\Br(\mathcal{X})\ar[r]\ar[d] & \Br(X_{\bar{K}})^{G_{K}} \ar[d] \\
\Br(\mathcal{X}_1)^G\ar[r]& (\Br(X_1\otimes _K\bar{K})^{G_K})^G}
\end{displaymath}
\end{proof}
\begin{rem}\label{rmkforthm3.2}
The strictly semi-stale assumption in the proof of Proposition \ref{cor1} can be removed by the same argument as above.
\end{rem}

\subsection{Proof of Theorem \ref{0cycle}}
\begin{thm}
There is a commutative diagram
\begin{displaymath}
\xymatrix{ A_0(X)\ar[r] \ar[rd]^{\phi_X} & \Hom(\Br(X)/\Br(K)+\Br(\CX), \QQ/\ZZ) \ar[d]^{\psi_X}\\
& \Alb_X(K)}
\end{displaymath}
and $\psi_X$ has a finite kernel and a finite cokernel. Let $f$ (resp. $g$) denote the natural map $\Br(\CX)\rightarrow \Br(X_{\bar{K}})^{G_K}$ (resp. $\Br(X)\rightarrow \Br(X_{\bar{K}})^{G_K}$) and $F$ denote the cokernel of $\Pic(X_{\bar{K}})^{G_K}\rightarrow \NS(X_{\bar{K}})^{G_K}$. Then there is a long exact sequence of finite groups
$$0\lra F\lra \Coker(\psi_X)^\vee  \lra \Ker(f)/\Br(\CX)\cap \Br(K) \lra H^1(K, \NS(X_{\bar{K}}))\\$$
$$\lra \Ker(\psi_{X})^\vee \lra \Coker(f)\lra \Coker(g) \lra 0.$$
\end{thm}
\begin{proof}
The commutativity of the diagram was proved by Kai (cf. \cite[Prop. 3.4]{Kai}). We will construct the long exact sequence by diagram chasing. By the Hochschild-Serre spectral sequence
$$
H^p(K,H^q(X_{K^s},\mathbb{G}_m))\Rightarrow H^{p+q}(X,\mathbb{G}_m)
$$
and the fact $H^3(K,\GG_m)=0$ (cf. \cite[Chap. II, Prop. 1.5]{Mil1}), there is an exact sequence
$$\Br(K)\lra \Ker(\Br(X)\lra \Br(X_{\bar{K}})^{G_K})\lra H^1(K, \Pic_{X/K})\lra 0.
$$
Consider the following diagram
\begin{displaymath}
\xymatrix{
0 \ar[r]& \Br(\CX)/\Br(K)\cap \Br(\CX)\ar[r]\ar[d]& \Br(X)/\Br(K)\ar[d] \ar[r] &\Br(X)/\Br(K)+\Br(\CX) \ar[r] \ar[d] & 0 \\
0 \ar[r]& \Br(X_{\bar{K}})^{G_K}\ar[r]^{\mathrm{id}}& \Br(X_{\bar{K}})^{G_K} \ar[r] & 0 \ar[r] & 0}
\end{displaymath}
By the Snake Lemma, there is a long exact sequence
$$0\lra \Ker(f)/\Br(K)\cap \Br(\CX) \lra H^1(K, \Pic_{X/K}) \lra \Br(X)/\Br(K)+\Br(\CX)$$
$$ \lra \Coker(f) \lra \Coker (g) \lra 0.$$
Consider the following diagram
\begin{displaymath}
\xymatrix{
& 0 \ar[d] \ar[r]& H^1(K,\Pic^0_{X/K})\ar[d] \ar[r]^{id} &H^1(K,\Pic^0_{X/K})\ar[r] \ar[d]^{a} & 0 \\
0 \ar[r]& \Ker(f)/\Br(K)\cap \Br(\CX) \ar[r]& H^1(K,\Pic_{X/K}) \ar[r]^{b} & \Br(X)/\Br(K)+\Br(\CX) & }
\end{displaymath}
Since $H^2(K, \Pic^0_{X/K})=0$ ( cf. \cite[Cor. 3.4]{Mil1}), there is an exact sequence
$$0\lra F\lra H^1(K,\Pic^0_{X/K})\lra H^1(K,\Pic_{X/K}) \lra H^1(K, \NS(X_{\bar{K}}))\lra 0.
$$
By the snake Lemma, we get a long exact sequence
$$0\lra F\lra \Ker(a)  \lra \Ker(f)/\Br(\CX)\cap \Br(K) \lra H^1(K, \NS(X_{\bar{K}}))
$$
$$\lra \Coker(a) \lra \Coker(b)\lra 0.$$
By the previous long exact sequence, we have
$$0\lra \Coker(b)\lra \Coker(f) \lra \Coker(g) \lra 0.$$
Thus, we have a long exact sequence
$$0\lra F\lra \Ker(a)  \lra \Ker(f)/\Br(\CX)\cap \Br(K) \lra H^1(K, \NS(X_{\bar{K}}))
$$
$$
\lra \Coker(a)\lra \Coker(f) \lra \Coker(g) \lra 0.$$
By definition, the map $\psi_X$ is the dual of the map $a$. So we have
$$\Coker(\psi_X)^\vee\cong \Ker(a) \quad \mathrm{and} \quad \Ker(\psi_X)^\vee\cong \Coker(a).$$ 
Since $F$ is a finitely generated torsion abelian group, $F$ is finite. There exists a finite Galois extension $L/K$ such that $G_L$ acts trivially on $\NS(X_{\bar{K}})$. Consider the exact sequence of Galois cohomology
$$0 \lra H^1(\Gal(L/K), \NS(X_{\bar{K}})) \lra H^1(K, \NS(X_{\bar{K}})) \lra H^1(L, \NS(X_{\bar{K}})),$$
the first group is finite since it is a finitely generated torsion abelian group. By the local class field theory, the third group is also finite. Thus,
$H^1(K, \NS(X_{\bar{K}}))$ is finite. By Corollary \ref{big}, $\Ker(f)$ and $\Coker(f)$ are finite. It follows that all groups in the long exact sequence are finite.

\end{proof}

\section{Proof of Theorem \ref{mainthm}}

\begin{lem}\label{purity}
Let $\pi:\mathcal{X}\longrightarrow S=\Spec R$ be a flat morphism with $\CX$ integral and regular, where $R$ is a DVR of characteristic $0$ with a fraction field $K$. Let $X$ denote the generic fiber of $K$. Then there is an exact sequence
$$
0\lra \Br(\CX)\lra \Br(X)\lra \Br(X_{K^h})/\Br(\CX_{R^h}).
$$
\end{lem}
\begin{proof}
Let $Y$ denote the special fiber of $\pi$. 
It suffices to show 
$$ H^3_{Y}(\CX,\mathbb{G}_m)\cong H^3_{Y}(\CX_{R^h},\mathbb{G}_m).$$
This follows from the Excision theorem (cf.\cite[Chap III, Prop. 1.27 and Cor. 1.28]{Mil3}).
\end{proof}
\begin{rem}\label{compl}
In fact, the exact sequence above still holds if we replace $R^h$ by its completion. This actually follows from purity of Brauer groups. We may assume that $Y$ is of codimension $1$ with a generic point $y$. Then there is an injection $\Br(X)/\Br(\CX)\hookrightarrow \Br(K(X)^{h}_y)/\Br(\CO^h_{\CX,y})$. Since $\Char(K(X))=0$, we have $\Br(K(X)^{h}_y)=\Br(\widehat{K(X)}_y)$ and $\Br(\CO^h_{\CX,y})=\Br(\widehat{\CO_{\CX,y}})$. It is easy to see that the completion does not change if we replace $R$ by $\hat{R}$.

\end{rem}

\begin{lem}\label{klem}
Let $\pi:\CX\lra C=\Spec \CO_K$ be proper flat morphism with $\CX$ regular, where $K$ is a number field. Assuming that the generic fiber $X$ of $\pi$ is geometrically connected over $K$, then the natural map
$$\Br(X)/(\Br(\CX)+\Br(K)) \lra \prod_{v\in C^{\circ}}\Br(X_{K^h_v})/(\Br(\CX_{O_v^{h}})+\Br(K_v^h))
$$ 
has a finite kernel and the group $\Br(\CX)\cap \overline{\Br(K)}$ is finite, where $C^\circ$ denotes the set of closed point of $C$, $\CO_v^h$ denotes the Henselian local ring of $C$ at $v$ and $\overline{\Br(K)}$ denotes the image of the pullback map $\Br(K)\lra \Br(X)$.
\end{lem}
\begin{proof}
Consider the following commutative diagram ( the existence of the columns follows from \cite[Lem. 1.16, Chap. III]{Mil3})
\begin{displaymath}
\xymatrix{ 
   &\Br(K) \ar[r]\ar[d] &\Br(X)\ar[r]\ar[d]^b& \Br(X)/\Br(K)\ar[r]\ar[d]^c & 0 \\
  &\oplus_{v}\Br(K_v^{h})\ar[r]^-{f}&\oplus_{v} \Br(X_{K^h_v})/\Br(\CX_{\CO_v^{h}})\ar[r] & \oplus_{v} \Br(X_{K^h_v})/(\Br(\CX_{\CO_v^{h}})+\Br(K_v^h)) & }
\end{displaymath}
 Let $a$ denote the map
$$\Br(K)\lra \bigoplus_{v}\Br(K_v^{h})/\Ker(f).$$
By the snake lemma, there is a long exact sequence
$$\Ker(a)\lra \Ker(b) \lra \Ker(c) \lra \Coker(a) \stackrel{\bar{f}}{\lra}\Coker(b).$$
By the purity of Brauer groups (cf. \cite{Ces} ) and Lemma \ref{purity}, $\Ker(b)=\Br(\CX)$. Thus, it suffices to show that $\Ker(a)$ and $\Ker(\bar{f})$ are finite groups. Since 
$$\Ker(\Ker(a)\lra \bigoplus_{ v|\infty} \Br(K_{v}) \oplus \Ker(f))=\Ker(\Br(K) \lra \bigoplus_{v}\Br(K_{v})),$$
by the exact sequence coming from the global class field theory (cf. \cite[Chap. II, Prop. 2.1]{Mil1})
$$
 0\longrightarrow \Br(K) \longrightarrow \bigoplus_{v} \Br(K_v)\longrightarrow \QQ/\ZZ \longrightarrow 0,
$$
we have that the map
$$\Ker(a)\lra \bigoplus_{ v|\infty} \Br(K_{v})\oplus \Ker(f)$$
is injective and $\Coker(a)$ is of cofinite type. Thus, it suffices to show that $\Ker(f)$ is finite and $\Coker(\bar{f})$ is of finite exponent. Let $L/K$ be a finite Galois extension such that $X(L)$ is not empty. Let $P\in X(L)$. Then $P$ defines a $K$-morphism $g:Z:=\Spec L \rightarrow X$ and the morphism can be extended to a $C$-morphism $g:\CZ:=\Spec \CO_L \lra \CX$ since $\pi$ is proper. Consider the composition
$$\bigoplus_{v}\Br(K_v^{h})\stackrel{f}{\lra} \bigoplus_{v} \Br(X_{K^h_v})/\Br(\CX_{\CO_v^{h}})\lra \bigoplus_{v} \Br(Z_{K^h_v})/\Br(\CZ_{\CO_v^{h}}).$$
Obviously, $\Ker(f)$ is a subgroup of the kernel of the composition. We want to show that the kernel of the composition is killed by $[L:K]$. Then, it will imply that $\Ker(f)$ is killed by $[L:K]$. The natural map $\CZ\lra C$ is finite flat of degree $[L:K]$. Its base change $Z_{K_v^h}\lra \Spec K_v^h$ is also finite flat of degree $[L:K]$. They induce restriction map $\mathrm{res}_v:\Br(K_v^h)\rightarrow \Br(Z_{K_v^h})$ and corestriction map $\mathrm{cores}_v:\Br(Z_{K_v^h})\rightarrow \Br(K_v^h)$ (cf. \cite[\S 3.8, p. 98]{CTS2} ). Since $\Br(\CO_v^h)=0$ for all $v\in C^\circ$, the composition map above can be identified with the restriction map
$$\bigoplus_{v}\Br(K_v^{h})/\Br(\CO_v^h)\stackrel{\oplus_v \mathrm{res}_v}{\lra}\bigoplus_{v} \Br(Z_{K^h_v})/\Br(\CZ_{\CO_v^{h}}).$$
Since the composition $(\oplus_v \mathrm{cores}_v)\circ(\oplus_v \mathrm{res}_v)$ is equal to the multiplication by $[L:K]$, thus, the kernel of $\oplus_v \mathrm{res}_v$ is killed by $[L:K]$. This proves that $\Ker(f)$ is killed by $[L:K]$. By functoriality, the same argument can imply that $\Ker(\bar{f})$ is also killed by $[L:K]$. To show that $\Ker(f)$ is finite, it suffices to show that for all but finitely many $v\in C^\circ$, the map
$$\Br(K_v^h)\lra \Br(X_{K^h_v})/\Br(\CX_{\CO_v^{h}})$$
is injective. If $X(K_v^h)$ is not empty, the restiction-corestriction argument above will imply that the kernel is killed by $1$ and therefore is trivial. Thus, it remains to be shown that $X(K_v^h)$ is not empty for all but finitely many $v$. Note that if $\CX_{\CO_v^h} \lra \Spec \CO_v^h$ is smooth and the special fiber over $v$ admits a rational point, by Hensel's lemma, $\CX_{O_v^h}$ admits a section. This will imply that $X(K_v^h)$ is not empty. Thus, it suffices to show that for all but finitely many $v$, the special fiber over $v$ is smooth and has a rational point. For the proof of this, we refer it to \cite[Chap. II, Lem. 2.1 , \S2]{CTSDV}.
\end{proof}
\begin{rem}
In fact, the above proof shows that the kernel is killed by $[L:K]$, and that $\Br(\CX)\cap \overline{\Br(K)}$ is killed by $2[L:K]$ for any finite extension $L/K$ such that $X(L) \neq \emptyset$.

If instead of assuming that $X$ is geometrically irreducible, we assume that $X$ is irreducible, then statement should be modified as follows:
the natural map
$$\Br(X)/(\Br(\CX)+H^2(K,\pi_{*}\GG_m)) \lra \prod_{v\in C^{\circ}}\Br(X_{K^h_v})/(\Br(\CX_{O_v^{h}})+H^2(K_v^h,\pi_{*}\GG_m))
$$ 
has a finite kernel. To explain this in more detail, let $L$ denote the algebraic closure of $K$ in $K(X)$. The Stein factorization $\CX \rightarrow \Spec \CO_L$ has a geometrically connected generic fiber. The sheaf $\pi_*\GG_m$ is the same as the direct image of $\GG_m$ under the map $\Spec L\rightarrow \Spec K$. So $H^2(K,\pi_{*}\GG_m)=\Br(L)$ and 
$H^2(K_v^h,\pi_{*}\GG_m)=\Br(L\otimes_K K_v^h)$.
\end{rem}
\begin{thm}\label{almostdonethm}
Let $\pi:\mathcal{X}\longrightarrow C$ be a proper flat morphism, where $C$ is $\Spec \CO_K$ for some number field $K$. Assume that $\mathcal{X}$ is regular and the generic fiber $X$ of $\pi$ is projective and geometrically connected over $K$. Let $\ell$ be a prime number. Then there are exact sequences up to finite groups
$$
0\longrightarrow \Sha(\Pic_{X/K}^0)\longrightarrow \Br(\mathcal{X}) \longrightarrow  \Br(X_{\olsi{K}})^{G_K},
$$
and
$$
0\longrightarrow \Sha(\Pic_{X/K}^0)(\ell)\longrightarrow \Br(\mathcal{X})(\ell)\longrightarrow \Br(X_{\olsi{K}})^{G_K}(\ell)\longrightarrow 0.
$$
\end{thm}

\begin{proof}
By the Hochschild-Serre spectral sequence

$$E_2^{p,q}=H^p(K, H^q(X_{\olsi{K}},\GG_m))\Rightarrow H^{p+q}(X,\GG_m)$$
and the fact $H^3(K,\GG_m)=0$ (cf. \cite[Prop. 2.7]{Mil4}), we get a long exact sequence
$$\Br(K)\lra \Ker(\Br(X)\lra \Br(X_{\olsi{K}})^{G_K})\lra H^1(K, \Pic_{X/K})\lra 0.
$$
It gives an injective natural map 
$$H^1(K, \Pic_{X/K})\lra \Br(X)/\Br(K).$$
Similarly, we get
$$H^1(K_v^h, \Pic_{X/K})\lra \Br(X_{K_v^h})/\Br(K_v^h).$$
Consider the following commutative diagram with exact rows

\begin{displaymath}
\xymatrix{ 
   0\ar[r]&H^1(K, \Pic_{X/K}) \ar[r]\ar[d] &\Br(X)/\Br(K)\ar[r]^b\ar[d]& \Br(X_{\olsi{K}})^{G_K}\ar[d] & \\
 &\prod_{v}H^1(K_v^h, \Pic_{X/K})\ar[r]^-{a}&\prod_{v} \Br(X_{K^h_v})/(\Br(\CX_{\CO_v^{h}})+\Br(K_v^h))\ar[r]&\prod_{v} \Br(X_{\olsi{K}})^{G_v}/\Br(\CX_{\CO_v^{h}}) &}
\end{displaymath}
where $G_v$ denotes $\Gal(\olsi{K}/K^h_v)$. The above diagram gives the following commutative diagram\\
\begin{tikzcd}[column sep=tiny]
0\arrow[r] &H^1(K, \Pic_{X/K}) \arrow[r]\arrow[d,"f"] &\Br(X)/\Br(K)\arrow[r]\arrow[d,"g"]& \Im(b)\arrow[d, "h"]\arrow[r] & 0\\
 0 \arrow[r]&\prod_{v}H^1(K_v^h, \Pic_{X/K})/\Ker(a)\arrow[r]&\prod_{v} \Br(X_{K^h_v})/(\Br(\CX_{\CO_v^{h}})+\Br(K_v^h))\arrow[r]&\prod_{v} \Br(X_{\bar{K}})^{G_v}/\Br(\CX_{\CO_v^{h}})
\end{tikzcd}
By the snake lemma, we get a long exact sequence
$$0\lra \Ker(f)\lra\Ker(g)\lra \Ker(h)\lra \Coker(f).$$
By Lemma \ref{klem}, the natural map
$$\Br(\CX)\lra\Ker(g)$$
has a finite kernel and a finite cokernel. Thus, to prove the theorem, it suffices to show that the natural maps
$$\Sha(\Pic^0_{X/K})\lra \Ker(f)$$
and
	$$\Ker(h)(\ell)\lra\Br(X_{\Bar{K}})^{G_K}(\ell)$$
have finite kernels and cokernels, and that the natural map
$$\Ker(h)\lra \Coker(f)$$
has a finite image. By Lemma \ref{Kera} below, $\Ker(a)$ is of finite exponent. Since there is an exact sequence
$$0\lra \Sha(\Pic_{X/K})\lra \Ker(f) \lra \Ker(a),$$
and $\Ker(f)$ is of cofinite type ( this follows from the fact that $\Ker(g)$ is of cofinite type ), the injective map
$$\Sha(\Pic_{X/K})\lra \Ker(f)$$
has a finite cokernel. By Lemma \ref{twosha}, the natural map 
$$\Sha(\Pic^0_{X/K})\lra\Sha(\Pic_{X/K})$$
has a finite kernel and a finite cokernel. Thus, the map
$$\Sha(\Pic^0_{X/K})\lra \Ker(f)$$
also has a finite kernel and a finite cokernel. Since $\Im(b)$ is of cofinite type, by Corollary \ref{big}, the target of $h$ is a product of finite groups. Hence, $h$ will map the maximal divisible subgroup of $\Im(b)$ to zero. So the divisible part of $\Im(b)$ is contained in $\Ker(h)$. By a theorem of Colliot-Thélène and Skorobogatov \cite{CTS1}, $\Im(b)$ has a finite index in $\Br(X_{\olsi{K}})^{G_K}$. Thus, they have the same maximal divisible subgroup. So we have 
$$\Ker(h)_{\div}=(\Br(X_{\olsi{K}})^{G_K})_{\div}.$$
It follows that the inclusion
$$\Ker(h)(\ell)\hookrightarrow \Br(X_{\olsi{K}})^{G_K}(\ell)$$
has a finite cokernel for any prime $\ell$. It remains to show that the natural map
$$\Ker(h)\lra \Coker(f)$$
has an image of finite exponent. To prove this, we will use the pullback trick. The idea is to find finitely many proper flat morphisms $\pi_i:\CZ_i\lra C$ where each $\CZ_i$ an integral regular scheme of dimension 2, along with $C$-morphisms $\CZ_i\lra \CX$. Then, we use the functoriality of the map $\Ker(h)\lra \Coker(f)$ to get a commutative diagram 
\begin{displaymath}
\xymatrix{\Ker(h)\ar[r]\ar[d] &\Coker(f)\ar[d]\\
	\bigoplus_i\Ker(h_i)\ar[r] &\bigoplus_i\Coker(f_i)}
\end{displaymath}
Since $\Ker(h_i)=0$ ( $\Br(Z_{\olsi{K}})^{G_K}=0$ ), to prove the claim, it is enough to find $\CZ_i$ such that the second column has a kernel of finite exponent. Since $\Ker(a_i)$ is of finite exponent, by the snake lemma, it suffices to to find $\CZ_i$ such that the natural map
$$\Coker(H^1(K, \Pic_{X/K})\rightarrow \prod_{v} H^1(K_v^h, \Pic_{X/K}))\rightarrow \bigoplus_i\Coker(H^1(K, \Pic_{Z_i/K})\rightarrow \prod_{v}H^1(K_v^h, \Pic_{Z_i/K}))$$
has a kernel of finite exponent. By Lemma \ref{coker} below, there exist smooth projective integral curves $Z_i \subset X$ satisfying the above condition. To get $\CZ_i$, we can take the Zariski closure of $Z_i$ in $\CX$ first and then desingularize it. Since such $Z_i$ may not be geometrically connected, in all discussions above, we need to replace $\Br(K)$ by $\Br(L)$ and $\Br(K_v^h)$ by $\Br(L\otimes_K K_v^h)$, where $L$ is the algebraic closure of $K$ in $K(Z_i)$. This completes the proof of the theorem.

\end{proof}
\begin{rem}\label{bigrem}
The above proof actually shows that $\Br(\CX)\lra \Ker(h)$ has a cokernel of finite exponent. Since $\Im(b)$ has a finite index in $\Br(X_{\olsi{K}})^{G_K}$, the natural map
$$\Br(X_{\olsi{K}})^{G_K}/\Br(\CX) \lra \prod_{v} \Br(X_{\olsi{K}})^{G_v}/\Br(\CX_{\CO_v^{h}})$$
has a finite kernel. To complete the proof of Theorem \ref{mainthm}, it suffices to show that for $\ell \gg 0$,
$$(\Br(X_{\olsi{K}})^{G_v}/\Br(\CX_{\CO_v^{h}}))(\ell)=0$$ 
for all finite places $v$. 
\end{rem}

\begin{lem}\label{Kera}
Let $a$ denote the map
$$\prod_{v}H^1(K_v^h, \Pic_{X/K}) \lra \prod_{v} \Br(X_{K^h_v})/(\Br(\CX_{\CO_v^{h}})+\Br(K_v^h)),$$
then $\Ker(a)$ is of finite exponent.
\end{lem}
\begin{proof}
By Lemma \ref{thekeytrick} (ii), there exist smooth projective integral curves $Z_i\subset X$ over $K$ and a continuous $G_K$-module $B$ with a $G_K$-equivariant map $B\ra\oplus_{i=1}^{m}\Pic(Z_{i,\olsi{K}})$ such that the induced $G_K$-equivariant map
$$\Pic(X_{\olsi{K}})\times B \lra \bigoplus_i\Pic(Z_{i,\olsi{K}})$$
has a kernel and a cokernel killed by some positive integer $N$. Taking the Zariski closure of $Z_i$ in $\CX$ and then desingularizing it, we get a $C$-morphism $\CZ_i\lra \CX$. Let $L_i$ denote the algebraic closure of $K$ in $K(Z_i)$. This gives a commutative diagram
\begin{displaymath}
\xymatrix{ 
H^1(K_v^h, \Pic_{X/K})\ar[r]^-{a_v}\ar[d]& \Br(X_{K^h_v})/(\Br(\CX_{\CO_v^{h}})+\Br(K_v^h))\ar[d]\\
\bigoplus_i H^1(K_v^h, \Pic_{Z_i/K})\ar[r]&\bigoplus_{v} \Br(Z_{i,K^h_v})/(\Br(\CZ_{i,\CO_v^{h}})+\Br(L_i\otimes_K K_v^h))}
\end{displaymath}
Since $\Br(\CZ_{i, \CO_v^{h}})=0$ (cf. \cite[Lem. 2.6]{Mil4}), by definition, the second row is injective. It follows that $\ker(a_v)$ is contained in the kernel of the first column, which is a subgroup of the kernel of the map
$$H^1(K_v^h, \Pic_{X/K})\oplus H^1(K, B)\lra \bigoplus_i H^1(K_v^h, \Pic_{Z_i/K}).$$
By the long exact sequence in Galois cohomology, the kernel of this map is killed by $N^2$. So $N^2\Ker(a)=0$.
\end{proof}
\begin{lem}\label{coker}
Let $X$ be a smooth projective geometrically connected variety over a number field $K$. Then, there exist smooth projective integral curves $Z_i \subset X$ such that the induced map
$$\Coker(H^1(K, \Pic_{X/K})\rightarrow \prod_{v} H^1(K_v^h, \Pic_{X/K}))\rightarrow \bigoplus_i\Coker(H^1(K, \Pic_{Z_i/K})\rightarrow \prod_{v}H^1(K_v^h, \Pic_{Z_i/K}))$$
has a kernel of finite exponent.
\end{lem}
\begin{proof}
We will use the similar argument as in the proof of the above lemma. Let $Z_i$ and $B$ as in the above lemma.  There is a $G_K$-equivariant map
$$\Pic(X_{\olsi{K}})\times B \lra \bigoplus_i\Pic(Z_{i,\olsi{K}})$$
with a kernel and a cokernel killed by some positive integer $N$. Thinking
$$\Coker(H^1(K, -)\lra \prod_{v} H^1(K_v^h, -))$$
as a functor, it commutes with finite direct sums. Thus, it suffices to show that for any morphism of continuous $G_K$-modules $P_1\lra P_2$ with a kernel and cokernel killed by $N$, the induced map
$$\Coker(H^1(K, P_1)\rightarrow \prod_{v} H^1(K_v^h, P_1))\rightarrow \Coker(H^1(K, P_2)\rightarrow \prod_{v}H^1(K_v^h, P_2))$$
has a kernel of finite exponent. Consider the following commutative diagram
\begin{displaymath}
\xymatrix{ 
&H^1(K, P_1)\ar[r]^a\ar[d]^f& H^1(K, P_2)\ar[d]^g\ar[r]&\Coker(a)\ar[r] \ar[d]^{\bar{g}}& 0\\
\Ker(b)\ar[r] &\prod_{v} H^1(K_v^h, P_1)\ar[r]^b &\prod_{v} H^1(K_v^h, P_2)\ar[r]&\Coker(b)&}
\end{displaymath}
By the assumption, $\Ker(b)$ and $\Coker(a)$ are killed by $N^2$. By the snake lemma, there is an exact sequence
$$\Ker(\bar{g})\lra \Coker(f)/\Ker(b)\lra \Coker(g).
$$
It follows that $\Coker(f)\lra \Coker(g)$ has a kernel killed by $N^4$.
\end{proof}
\bigskip

By Remark \ref{compl}, we can replace the Henselain local ring with its completion in the proof of Theorem \ref{almostdonethm}. If we exclude finitely many places, the map $\CX_{\CO_v} \lra \Spec \CO_v$ will be a smooth projective morphism. Using the pullback trick as in the proof of Proposition \ref{plocalsurj}, one can show that
$$\Br(\CX_{\CO_v})\lra\Br(\CX_{\CO_v^{sh}})^{G_{k(v)}}$$
has a cokernel killed by some positive integer independent of $v$. In conjunction with Lemma \ref{lpart}, this implies that for all but finitely many $v$, the natural map
$$\Br(\CX_{\CO_v})(\non p_v)\lra\Br(X_{\olsi{K_v}})^{G_{K_v}}(\non p_v)$$
has a cokernel killed by some positive integer independent of $v$, where $p_v$ denotes the characteristic of the residue field at $v$. Thus, to prove Theorem \ref{mainthm}, by Remark \ref{bigrem}, it suffices to show that the natural map
$$\Br(\CX_{\CO_v})(p_v)\lra \Br(X_{\olsi{K_v}})^{G_{K_v}}(p_v)$$
is surjective for all but finitely many $v$.  We will establish this in the Corollary \ref{lastcor} below. To prove Corollary \ref{lastcor},
we need an integral version of the local invariant cycle theorem.

\begin{prop}
Let $K/\QQ_p$ be an unramified finite extension, and let $\pi: \CX\lra \Spec \CO_K$ be a smooth projective morphism with generic fiber $X$. Assuming $p\geq 5$, then the natural map
$$H^2_{\fppf}(\CX, \mu_{p^n})\lra H^2(X_{\olsi{K}},\mu_{p^n})^{G_K}
$$
is surjective for any $n\geq 1$.
\end{prop}
\begin{proof}
Let $\CX_n$ denote the $\CO_K/p^n$-scheme $\CX\otimes_{\ZZ}\ZZ/p^n$. Fontaine and Messing \cite[III. 3.1]{FM} defined a sheaf $\CS_n(r)$ on the small syntomic site $(\CX_n)_{syn}$ of $\CX_n$. Following definitions and notations in \cite[\S 2.1]{EN}, for $0\leq r\leq p-1$, define
$$\CS_n(r):=\Ker(\CJ_n^{[r]}\stackrel{1-\varphi_r}{\lra}\CO_n^{cr}).$$
There is a short exact sequence of sheaves on $(\CX_n)_{syn}$
\begin{equation}
0 \longrightarrow \CS_n(r) \longrightarrow \CJ_n^{[r]} \stackrel{1-\varphi_r}{\longrightarrow} \CO_n^{cr} \longrightarrow 0. \label{theshort}
\end{equation}
For our purpose, we only need to consider the case with $r=1$. The syntomic cohomology $H^i((\CX_n)_{syn},\CS_n(1))$ computes the flat cohomology
$H^i_{\fppf}(\CX, \mu_{p^n})$ (cf. \cite{Kur} and \cite[\S 1.4 and \S 4.5]{Sat} ). The natural map 
$$H^i_{\fppf}(\CX,\mu_{p^n}) \lra H^i(X,\mu_{p^n})$$
is compatible with the Fontaine-Messing map 
$$H^i((\CX_n)_{syn},\CS_n(1))\lra H^i(X,\mu_{p^n}).$$
Thus, it suffices to show that the following map induced by the Fontaine-Messing map 
$$H^2((\CX_n)_{syn},\CS_n(1))\lra H^2(X_{\olsi{K}},\mu_{p^n})^{G_K}$$
is surjective for any $n\geq 1$.\\
Let $Y$ denote the special fiber $\CX_1$. Following notations in \cite[III. 4.10]{Nek}, set
$$M^2_n:=H^2((\CX_n)_{syn},\CO_n^{cr})= H^2((Y/W_n)_{\cris},\CO_{Y/W_n})=H^2_{dR}(\CX_n/W_n),$$
$$F^rM_n^2:=H^2((\CX_n)_{syn},\CJ_n^{[r]})=H^2((\CX_n)_{\zar}, \sigma_{\geq r}\Omega^\bullet_{\CX_n/W_n}), \quad T_n^2:=H^2(X_{\olsi{K}}, \ZZ/p^n\ZZ).$$
By \cite[II. 2.7]{FM}, $(M^2_n, F^rM_n^2, \varphi_r)$ defines an object of  $MF_{W, tors}^{[0,2]}\subseteq MF_{W, tors}^{[0,p-1[}$ (cf. \cite[III. Prop. 4.11]{Nek} or \cite[Thm. 3.2.3]{BM} for details ). For $0\leq r<p-1$, by \cite[III. 4.8]{Nek},
the functor $T$ (cf. \cite[III. 4.6]{Nek} for the definition) induces an isomorphism
$$\alpha_{r,M^2_n}:(F^rM^2_n)^{\varphi_r=1} \lra (T_n^2(r))^{G_K}.$$
Taking cohomology for (\ref{theshort}), we get a long exact sequence
$$H^2((\CX_n)_{syn},\CS_n(1))\lra H^2((\CX_n)_{syn},\CJ_n^{[1]}) \stackrel{1-\varphi_1}\lra H^2((\CX_n)_{syn},\CO_n^{cr})$$
$$\lra H^3((\CX_n)_{syn},\CS_n(1)).$$
It gives a surjective map
$$H^2((\CX_n)_{syn},\CS_n(1))\lra H^2((\CX_n)_{syn},\CJ_n^{[1]})^{\varphi_1=1}.$$
By \cite[III. Thm. 5.2]{Nek}, there is a commutative diagram
\begin{displaymath}
\xymatrix{
H^2((\CX_n)_{syn},\CS_n(1)) \ar[r]\ar[d]^\nu& H^2((\CX_n)_{syn},\CJ_n^{[1]})^{\varphi_1=1}\ar[d]^{\alpha_{1,M_n^2}}\\
H^2(X,\ZZ/p^n\ZZ(1))\ar[r]& H^2(X_{\olsi{K}},\ZZ/p^n\ZZ(1))^{G_K}}
\end{displaymath}
where $\nu$ denotes the Fontaine-Messing map. Since the first row is surjective and $\alpha_{1,M_n^2}$ is an isomorphism, the natural map
$$H^2((\CX_n)_{syn},\CS_n(1))\lra H^2(X_{\olsi{K}},\mu_{p^n})^{G_K}$$
is surjective. This completes the proof.
\end{proof}

\begin{cor}\label{lastcor}
Let $\pi:\mathcal{X}\longrightarrow \Spec \CO_K$ be a proper flat morphism, where $K$ is a number field. Assume that $\mathcal{X}$ is regular and the generic fiber $X$ of $\pi$ is projective and geometrically connected over $K$. For any finite place $v$, let $p_v$ denote the characteristic of the residue field of $\CO_K$ at $v$. Then, for all but finitely many places $v$, the natural map
$$\Br(\CX_{\CO_v})(p_v)\lra \Br(X_{\olsi{K_v}})^{G_{K_v}}(p_v)$$
is surjective.
\end{cor}
\begin{proof}
$K/\QQ$ only ramifies at finitely many places. Thus, by the proposition above,
for all but finitely many $v$, the natural map
$$H^2_{\fppf}(\CX_{\CO_v}, \mu_{p_v^n})\lra H^2(X_{\olsi{K_v}},\mu_{p_v^n})^{G_{K_v}}
$$
is surjective for any $n\geq 1$. By Proposition \ref{prop2.1}, for all but finitely many $v$, the natural map
$$H^2(X_{\olsi{K_v}},\mu_{p_v^n})^{G_{K_v}} \lra \Br(X_{\olsi{K_v}})^{G_{K_v}}[p_v^n]$$
is surjective for any $n\geq 1$. By the following commutative diagram
\begin{displaymath}
\xymatrix{
 H^2_{\fppf}(\CX_{\CO_v},\mu_{p^n_v})\ar[r]\ar[d] & \Br(\CX_{\CO_v})[p^n_v] \ar[d]\\
H^2(X_{\olsi{K_v}},\mu_{p^n_v})^{G_{K_v}}\ar[r]  &  \Br(X_{\olsi{K_v}})^{G_{K_v}}[p^n_v],}
\end{displaymath}
 for all but finitely many $v$, the natural map
$$\Br(\CX_{\CO_v})[p^n_v]\lra \Br(X_{\olsi{K_v}})^{G_{K_v}}[p^n_v]$$
is surjective for any $n\geq 1$.

\end{proof}

\begin{rem}
By  Theorem \ref{0cycle}, the corollary implies that  for all but finitely many places $v$, the $p_v$-primary part of the kernel of the natural map
$$\Hom(\Br(X_{K_v})/\Br(K)+\Br(\CX_{\CO_v}), \QQ/\ZZ)\lra \mathrm{Alb}_X(K_v)$$
is trivial.
\end{rem}

\section{Applications}
\subsection{Reduction of Artin's question}
\begin{lem}
Assuming that $\Br(X_{\olsi{K}})^{G_K}$ is finite for all smooth projective geometrically connected surfaces over a number field $K$, then $\Br(X_{\olsi{K}})^{G_K}$ is finite for all smooth projective geometrically connected varieties over $K$.
\end{lem}
\begin{proof}
Let $X$ be a smooth projective geometrically connected variety over a number field $K$. Assuming that $\dim(X)>2$, by \cite[Cor. 1.6.2.1]{Amb}, there exists a smooth projective geometrically connected hyperplane section $D$ of $X$ such that the induced map
$$\NS(X)\otimes_\ZZ \QQ\lra\NS(D) \otimes_\ZZ \QQ$$
is an isomorphism. By Proposition \ref{prop2.1}, for any prime $\ell$, there is a commutative diagram with exact rows
 \begin{displaymath}
\xymatrix{
 0 \ar[r] & \NS(X)\otimes_{\mathbb{Z}}\mathbb{Q}_\ell \ar[r]\ar[d] &H^2(X_{\olsi{K}},\mathbb{Q}_\ell(1))^{G_K}\ar[r] \ar[d] & V_\ell \Br(X_{\olsi{K}})^{G_K} \ar[r]\ar[d] & 0
\\
0 \ar[r] & \NS(D)\otimes_{\mathbb{Z}}\mathbb{Q}_\ell \ar[r] &H^2(D_{\olsi{K}},\mathbb{Q}_\ell(1))^{G_K}\ar[r] & V_\ell \Br(D_{\olsi{K}})^{G_K} \ar[r] & 0}
\end{displaymath}
Since the first column is an isomorphism and the second column is injective ( the weak Lefschetz theorem), by the snake lemma, the third column is injective. This implies that for any $\ell$, the $\ell$-primary part of the kernel of the map
$$\Br(X_{\olsi{K}})^{G_K} \lra\Br(D_{\olsi{K}})^{G_K}$$
is finite. Since $\NS(X)$ and $\NS(D)$ are finitely generated abelian groups, thus for sufficiently large $\ell$, the induced map
$$\NS(X)\otimes_\ZZ\mathbb{Q}_\ell/\mathbb{Z}_\ell\lra\NS(X)\otimes_\ZZ\mathbb{Q}_\ell/\mathbb{Z}_\ell$$
is an isomorphism. By the weak Lefschetz theorem, the induced map
$$H^2(X_{\olsi{K}},\mathbb{Q}_\ell/\ZZ_\ell(1))\lra H^2(D_{\olsi{K}},\mathbb{Q}_\ell/\ZZ_\ell(1))$$
is injective. By Proposition \ref{prop2.1}, we have a similar diagram as before for torsion coefficients, when $\ell$ is sufficently large. Thus, natural map
$$\Br(X_{\olsi{K}})^{G_K}(\ell) \lra\Br(D_{\olsi{K}})^{G_K}(\ell)$$
is injective for all but finitely many $\ell$. This proves that the kernel of
$$\Br(X_{\olsi{K}})^{G_K} \lra\Br(D_{\olsi{K}})^{G_K}$$
is finite. By the induction on the dimension of $X$, the claim follows.
\end{proof}
\begin{thm}
Assuming that $\Br(X_{\bar{K}})^{G_K}$ is finite for all smooth projective surfaces and $\Sha(\Pic^0_{X/K})$ is finite for all smooth projective geometrically connected curves, then $\Br(\CX)$ is finite for all regular proper flat schemes $\CX$ over $\ZZ$.
\end{thm}
\begin{proof}
Let $X$ be a smooth proper geometrically connected variety over a number field $K$. By resolution of singularity in charateristic $0$, there is a smooth projective variety $X^\prime$ and a birational morphism $X^\prime \lra X$. Since Brauer group is a birational inviants and $\Pic^0_{X/K}$ and $\Pic^0_{X^\prime/K}$ are isogenous to each other, we may assume that $X$ is projective. If $\dim(X)\geq 2$, taking a smooth hyperplane section $D$ in $X$, then the induced map $\Pic^0_{X/K}\lra\Pic^0_{D/K}$ has a finite kernel since $H^1(X_{\olsi{K}},\QQ_\ell)\lra H^1(D_{\olsi{K}},\QQ_\ell)$ is injective. Since $\Pic^0_{X/K}\lra \Pic^0_{D/K}$ has a finite kernel, there is an abelian variety $A/K$ such that $\Pic^0_{X/K}\times A$ is isogenous to $\Pic^0_{D/K}$( cf.\cite[\S 3.3, p17]{Yua} for details). Thus, the finiteness of $\Sha(\Pic^0_{X/K})$ can be reduced to curves. By the above lemma, the finiteness of $\Br(X_{\olsi{K}})^{G_K}$ can be reduced to surfaces. Let $X$ denote the generic fiber of $\CX$. By the assumption,  $\Sha(\Pic^0_{X/K})$ and $\Br(X_{\olsi{K}})^{G_K}$ are finite. By Theorem \ref{mainthm}, $\Br(\CX)$ is finite.

\end{proof}
\begin{thm}
Assuming that $\Br(\CX)$ is finite for all $3$-dimensional regular proper flat schemes over $\ZZ$, then $\Br(\CX)$ is finite for all regular proper flat schemes over $\ZZ$.
\end{thm}
\begin{proof}
 By the theorem above, it suffices to show that $\Br(X_{\olsi{K}})^{G_K}$ and $\Sha(\Pic^0_{X/K})$ are finite for all smooth projective geometrically connected surfaces over a number field $K$ under the assumption. By de Jong's theorem \cite[Cor. 5.1]{deJ1}, there is an alteration $X^\prime \lra X$ such that $X^\prime$ admits a flat proper regular integral model over the ring of integers of a number field. By extending $K$, we may assume that $X^\prime$ is smooth and geometrically connected over $K$. Then, $\Pic^0_{X/K}\lra\Pic^0_{X^\prime/K}$ has a finite kernel since $H^1(X_{\olsi{K}},\QQ_\ell)\lra H^1(X^\prime_{\olsi{K}},\QQ_\ell)$ is injective. Thus, the finiteness of $\Sha(\Pic^0_{X^\prime/K})$ implies the finiteness of $\Sha(\Pic^0_{X/K})$. Since $K(X^\prime_{\olsi{K}})$ is a finite extension over $K(X_{\olsi{K}})$, by the restriction-corestriction argument, the induced map $\Br(K(X_{\olsi{K}}))\lra 
\Br(K(X^\prime_{\olsi{K}}))$ has a kernel of finite exponent. It follows that
$$\Br(X_{\olsi{K}})^{G_K}\lra 
\Br(X^\prime_{\olsi{K}})^{G_K}$$
has a finite kernel. By the assumption and Theorem \ref{mainthm}, $\Br(X^\prime_{\olsi{K}})^{G_K}$ and $\Sha(\Pic^0_{X^\prime/K})$ are finite. Thus, $\Br(X_{\olsi{K}})^{G_K}$ and $\Sha(\Pic^0_{X/K})$ are also finite. This completes the proof.
\end{proof}

\subsection{Examples}
\begin{prop}
Let $X$ be a principal homogeneous space of an abelian variety over a number field $K$. Assuming that  $X/K$  admits a proper regular model  $\pi:\CX \lra C$ as in Theorem \ref{mainthm}, then there is an isomorphism up to finite groups
$$ \Sha(\Pic_{X/K}^{0})\cong \Br(\CX).
$$

\end{prop}
\begin{proof}
By \cite[Thm. 1.1]{SZ}, $\Br(X_{K^s})^{G_K}$ is finite. Then the claim follows directly from Theorem \ref{mainthm}.
\end{proof}
\begin{prop}(Tankeev)
Let $\pi:\CX \lra C$ a proper flat morphism as in Theorem \ref{mainthm}. Assuming that the generic fiber $X$ is a $K3$ surface, then $\Br(\CX)$ is finite.
\end{prop}
\begin{proof} 
 Since $H^1(X,\SO_X)=0$, we have $\Pic^0_{X/K}=0$. By \cite[Thm. 1.2]{SZ}, $\Br(X_{\olsi{K}})^{G_K}$ is finite. Then the claim follows directly from Theorem \ref{mainthm}. 
\end{proof}


\Addresses

\begin{thebibliography}{9}
\bibitem{AG1v2}
U. Görtz, T. Wedhorn, Algebraic Geometry I. Springer Studium Mathematik–Master, 2nd edn. Springer Spektrum, Wiesbaden (2020)

\bibitem{Amb}
E. Ambrosi, Specialization of N\'eron-Severi groups in positive characteristic, arXiv preprint arXiv:1810.06481 .
\bibitem{And} 
Y. Andr\'e, Pour une th\'eorie inconditionnelle des motifs, Publ. Math. IHES 83, p. 5-49, 1996.

\bibitem{BM}
 C. Breuil, W. Messing. Torsion étale and crystalline cohomologies. Astérisque 279 (2002): 81-124.

\bibitem{Ces}
 K. \v{C}esnavi\v{c}ius, Purity for the Brauer group. Duke 
       Mathematical Journal, 168(8)(2019), pp.1461-1486.
       
\bibitem{CN}
P. Colmez, W. Nizioł, Syntomic complexes and p-adic nearby cycles, Inventiones mathematicae 208, no. 1 (2017): 1-108.  
       
\bibitem{CT}
J.-L. Colliot-Thélène,  L'arithmétique du groupe de Chow des zéro-cycles, Journal de théorie des nombres de Bordeaux 7.1 (1995): 51-73.
\bibitem{CTSa}
J.-L. Colliot-Thélène, S. Saito, Z\'ero-cycles sur les vari\'et\'es $p$-adiques et groupe de Brauer, International Mathematics Research Notices 4 (1996): 151-160.
\bibitem{CTS1}
J.-L. Colliot-Thélène,  A.N. Skorobogatov, Descente galoisienne sur le groupe de Brauer. J. reine angew. Math. 682 (2013) 141-165.
\bibitem{CTS2}
J.-L. Colliot-Thélène,  A.N. Skorobogatov, The Brauer–Grothendieck group, \url{http://wwwf.imperial.ac.uk/~anskor/brauer.pdf}
\bibitem{CTSDV}
 J.-L. Colliot-Thélène, P. Swinnerton-Dyer, and P. Vojta,  Arithmetic Geometry: Lectures Given at the CIME Summer School Held in Cetraro, Italy, September 10-15, 2007. Springer, 2010.	

   
\bibitem{Del2}
   P. Deligne, ``Weil's conjecture: II." IH\'ES Mathematical Publications 52 (1980): 137-252.





\bibitem{deJ1}
  A. J. de Jong,  Families of curves and alterations. In Annals of the Fourier Institute (Vol. 47, No. 2, 1997, pp. 599-621).

\bibitem{deJ2}
 A. J. de Jong, Tate conjecture for divisors. Unpublished note.
\bibitem{DN}	
F. D\'eglise, W. Nizioł,  On p-adic absolute Hodge cohomology and syntomic coefficients, I. Commentarii Mathematici Helvetici. 2018.


\bibitem{EGA4}
A. Grothendieck,
Éléments de géométrie algébrique : IV. Étude locale des schémas
et des morphismes de schémas, Troisième partie
Publications mathématiques de l'I.H.É.S., tome 28 (1966), p. 5-255.

\bibitem{EN}
V. Ertl, W. Niziol, Syntomic cohomology and p-adic motivic cohomology. ALGEBRAIC GEOMETRY 6, no. 1 (2019): 100-131.






\bibitem{Fal}
G. Faltings, Crystalline cohomology and p-adic Galois representations. In: Igusa, J.I. (ed.) Algebraic Analysis, Geometry and Number Theory, pp. 25–80. Johns Hopkins University Press, Baltimore (1989)
\bibitem{FlMo}
M. Flach, B. Morin, 2012. On the Weil-\'Etale Topos of Regular Arithmetic Schemes. Documenta Mathematica, 17, pp.313-399.

\bibitem{FM}
J.-M. Fontaine, W. Messing, p-adic periods and p-adic \'etale cohomology, Contemporary Math. 67, 1987, 179-207.

\bibitem{Fu}
L. Fu, 2011. Etale cohomology theory (Vol. 13). World Scientific.

\bibitem{Fuj}    
 K. Fujiwara,
        A proof of the absolute purity conjecture (after Gabber).    In Algebraic geometry 2000, Azumino (pp. 153-183). Mathematical Society of Japan.
\bibitem{Gei1}
T. Geisser, Comparing the Brauer group to the Tate–Shafarevich group. Journal of the Institute of Mathematics of Jussieu, 1-6(2018).
\bibitem{Gei2}
T. Geisser,
 Tate's Conjecture and the Tate–Shafarevich Group over Global Function Fields, \textit{Journal of the Institute of Mathematics of Jussieu}, 2018, pp. 1–22.
\bibitem{Gei3}
T. Geisser, Hasse principles for étale motivic cohomology. Nagoya Mathematical Journal, 236, 63-83 (2019).
\bibitem{GeMo}
T. Geisser, B. Morin, On the kernel of the Brauer-Manin pairing, J. Number Theory 238 (2022), 444–463.
\bibitem{Goa}
C.D. Gonzalez-Aviles, Brauer Groups and Tate-Shafarevich Groups. J. Math. Sci. Univ. Tokyo, 2003, 10: 391-419.

\bibitem{Gro3} 
A. Grothendieck, Le groupe de Brauer. III. Exemples et compl\'ements, Dix Expos\'es sur la Cohomologie
des Sch\'emas, North-Holland, Amsterdam, 1968, pp. 88--188.
\bibitem{Har}
R. Hartshorne, Algebraic geometry. Vol. 52. Springer Science \& Business Media, 2013.



\bibitem{Illu}           
      L. Illusie, ``Crystalline cohomology." Motives (Seattle, WA, 1991) 55 (1994): 43-70.     

\bibitem{Jan}

     U. Jannsen, Weights in arithmetic geometry. Japanese Journal of Mathematics, 5(1)(2010), pp.73-102.
\bibitem{JS}
        U. Jannsen, S. Saito,
  Bertini theorems and Lefschetz pencils over discrete valuation rings, with applications to higher class field theory. Journal of Algebraic Geometry, 21(4)(2012), 683-705.

\bibitem{Kai}
W. Kai, A higher-dimensional generalization of Lichtenbaum duality in terms of the Albanese map, Compositio Mathematica 152, no. 9 (2016): 1915-1934.
\bibitem{Ka1}
K. Kato, On p-adic vanishing cycles (application of ideas of Fontaine-Messing). In: Algebraic Geometry, Sendai, 1985. Adv. Stud. Pure Math., vol. 10, pp. 207–251. North-Holland,Amsterdam (1987)

\bibitem{Ka2}
 K. Kato, Semistable reduction and p-adic \'etale cohomology. Ast\'erisque 223, 269–293(1994)

\bibitem{Kur}
M. Kurihara, A note on $p$-adic etale cohomology. Proceedings of the Japan Academy, Series A, Mathematical Sciences 63, no. 7 (1987): 275-278.
\bibitem{Lev}
M. Levine, K-theory and motivic cohomology of schemes, preprint(1999).

\bibitem{Lic}
S. Lichtenbaum, Duality theorems for curves over p-adic fields, Invent.Math. 7 (1969), 120—136.

\bibitem{Liu}
Q. Liu, Algebraic geometry and arithmetic curves. Vol. 6. Oxford Graduate Texts in Mathe, 2002
\bibitem{LLR1}
Q. Liu, D. Lorenzini and M. Raynaud, N\'eron models, Lie algebras, and reduction of curves of genus one, Invent. math. 157(2004), 455-518.

\bibitem{LLR2}
Q. Liu, D. Lorenzini and M. Raynaud, On the Brauer group of a surface, Invent. math. 159(2005), 673-676.





\bibitem{Man}
Y. I. Manin, Le groupe de Brauer-Grothendieck en géométrie diophantienne, in Actes du Congrès International des Mathématiciens (Nice, 1970), Tome 1, GauthierVillars, Paris, 1971, 401—411.


\bibitem{Mil1}
  J.S. Milne,
       Arithmetic duality theorems, Vol. 20. Charleston, SC: BookSurge, 2006.
\bibitem{Mil3}
 J.S. Milne, Etale cohomology (PMS-33)(1980) (Vol. 5657). Princeton university press.
\bibitem{Mil4}
J.S. Milne, Comparison of the Brauer Group with the Tate-Safarevic group,
J. Fac. Science Univ. Tokyo, Sec. IA 28 (1981), 735–743.     

\bibitem{Mor}
M. Morrow, A Variational Tate Conjecture in crystalline cohomology. Journal of the European Mathematical Society, 21(11)(2019), 3467-3511.
\bibitem{Nek}
J. Nekov\'a\v{r},
Syntomic cohomology and p-adic regulators, preprint (1998). \url{https://webusers.imj-prg.fr/~jan.nekovar/pu/syn.pdf}
\bibitem{NN}
J. Nekov\'a\v{r}, W. Nizioł,
  Syntomic cohomology and p-adic regulators for varieties over p-adic fields, Algebra \& Number Theory. 2016 Oct 7;10(8):1695-790.
\bibitem{OS}
M. Orr, A.N. Skorobogatov,  Finiteness theorems for K3 surfaces and abelian varieties of CM type. Compositio Mathematica, 2018, 154(8): 1571-1592.
\bibitem{Qin1}
Y. Qin, On the Brauer groups of fibrations. Mathematische Zeitschrift 307.1 (2024): 1-20.
\bibitem{Qin2}
Y. Qin, On geometric Brauer groups and Tate-Shafarevich groups. arXiv:2012.01681.
\bibitem{RZ}
M. Rapoport, T. Zink, Über die lokale Zetafunktion von Shimuravarietäten. Monodromiefiltration und
verschwindende Zyklen in ungleicher Charakteristik. Invent. Math. 68, 21–101 (1982).
\bibitem{SS1}
S. Saito, K. Sato, Zero-cycles on varieties over p-adic fields and Brauer groups, Ann. Sci. \'Ec. Norm. Sup\'er. (4) 47(3) (2014), 505-537.
\bibitem{SS2}
S. Saito,  K. Sato, and U. Jannsen, A finiteness theorem for zero-cycles over p-adic fields, Annals of mathematics (2010): 1593-1639.

\bibitem{Sat}
 K. Sato, p-adic \'etale Tate twists and arithmetic duality (with an appendix by Hagihara, K.). Ann. Sci. \'Ecole Norm. Sup. (4) 40, 519–588 (2007)
\bibitem{Sat1}
K. Sato,  Logarithmic Hodge–Witt sheaves on normal crossing varieties.  Mathematische Zeitschrift 257, no. 4 (2007): 707-743.
\bibitem{SZ}
A.N. Skorobogatov, Y.G. Zarkhin, A finiteness theorem for the Brauer group of Abelian varieties and KS surfaces. Journal of Algebraic Geometry, 17(3)(2008), 481-502.

\bibitem{Tan1}
S.G. Tankeev, On the Brauer group of an arithmetic scheme. Izvestiya: Mathematics, 65(2)(2001), 357.
\bibitem{Tan2}
S.G. Tankeev, On the Brauer group of an arithmetic scheme. II. Izvestiya: Mathematics, 67(5)(2003), 1007.
\bibitem{Tan3}
S.G. Tankeev, On the finiteness of the Brauer group of an arithmetic scheme. Mathematical Notes, 95(1-2)(2014), 121-132.
 \bibitem{Tat2}
 J. Tate, Conjectures on algebraic cycles in $l$-adic cohomology. Motives (Seattle, WA, 1991) 55 (1994): 71-83..
\bibitem{Ts1}
T. Tsuji, p-adic \'etale and crystalline cohomology in the semistable reduction case. Invent.
math. 137, 233–411 (1999).

\bibitem{Ts2}
 T. Tsuji, On p-adic nearby cycles of log smooth families. Bull. SMF 128, 529–575 (2000).
\bibitem{Ulm}
D. Ulmer, Curves and Jacobians over function fields, Arithmetic geometry over global function fields, 2014, pp.281-337.

\bibitem{Vidal}
 I. Vidal, 
Th\'eorie de Brauer et conducteur de Swan. Journal of Algebraic Geometry 13.2 (2004): 349-391.


\bibitem{Yua}
 X. Yuan, Comparison of arithmetic Brauer groups with geometric Brauer group. Preprint: arXiv:2011.12897v2 available at \url{https://arxiv.org/abs/2011.12897}.

	\end{thebibliography}
\end{document}